\documentclass{amsart}

\usepackage{amsfonts}
\usepackage{fancyhdr}
\usepackage{verbatim}
\usepackage{graphicx,color}
\usepackage{mathscinet}

\usepackage{amssymb,amsfonts}
\usepackage{xcolor}

\newtheorem{theorem}{Theorem}[section]
\newtheorem{thm}[theorem]{Theorem}
\newtheorem{lemma}[theorem]{Lemma}
\newtheorem{lem}[theorem]{Lemma}
\newtheorem{proposition}[theorem]{Proposition}

\newtheorem{corollary}[theorem]{Corollary}

\newtheorem{definition}[theorem]{Definition}

\newtheorem{remark}[theorem]{Remark}

\newtheorem{Lemma}[theorem]{Lemma}
\newtheorem{Corollary}[theorem]{Corollary}
\newtheorem{Definition}[theorem]{Definition}

\newtheorem{Remark}[theorem]{Remark}

\newcommand{\be}{\begin{equation}}
\newcommand{\ee}{\end{equation}}
\newcommand{\beq}{\begin{equation*}}
\newcommand{\eeq}{\end{equation*}}
\newcommand{\enq}{\end{equation}}
\newcommand{\ben}{\begin{eqnarray}}
\newcommand{\een}{\end{eqnarray}}
\newcommand{\bea}{\begin{eqnarray*}}
\newcommand{\eea}{\end{eqnarray*}}

\def\cH{{\mathcal H}}

\newcommand{\cL}{ {\mathcal{L}}}

\newcommand{\zb}{\overline{z}}

\newcommand{\clos}{\mbox{\rm clos}}

\def\ker{{\mathrm{ker\,}}}
\def\Ran{{\mathrm{Ran\,}}}

\def\Span{{\mathrm{Span\,}}}
\newcommand{\Rr}{{\mathbb{R}}}

\newcommand{\Eta}{\eta}

\newcommand{\mut}{\widetilde{\mu}}

\newcommand{\mutb}{\overline{\widetilde{\mu}}}

\newcommand{\llangle}{\left\langle}
\newcommand{\rrangle}{\right\rangle}

\newcommand{\eps}{\varepsilon}
\def\C{\mathbb C}
\def\R{\mathbb R}

\newcommand{\norm}[1]{\left\Vert#1\right\Vert}

\newcommand{\U}{\left(\begin{array}{c} v_{-} \\ u \\ v_{+} \end{array}\right)}

\def\@fnsymbol#1{\ensuremath{\ifcase#1\or \dagger\or \ddagger\or
   \mathsection\or \mathparagraph\or \|\or **\or \dagger\dagger
   \or \ddagger\ddagger \else\@ctrerr\fi}}
    \makeatother

\numberwithin{equation}{section}

\def\C{\mathbb C}
\def\R{\mathbb R}

\def\H{\mathcal H}
\def\L{\mathcal L}

\newcommand{\ian}[1]{{\textcolor{black}{#1}}}

\newcommand{\mm}[1]{{\textcolor{black}{#1}}}

\title[Spectral form]{The spectral form of the functional model for maximally dissipative operators: A Lagrange identity approach}

\author[Brown]{Malcolm Brown}
\address{Malcolm Brown, School of Computer Science and Informatics, Cardiff University, Abacws,
Senghennydd Road, Cardiff CF24 4AG, UK. \em Malcolm Brown died on the 14th of January 2022.}

\author[Marletta]{Marco Marletta}
\address{Marco Marletta, School of Mathematics, Cardiff University, Abacws,
Senghennydd Road, Cardiff CF24 4AG, UK}
\email{MarlettaM@cardiff.ac.uk}

\author[Naboko]{Serguei Naboko}
\address{Serguei Naboko, Dept. of Mathematics and Mathematical Physics,
Russia, 198904, St. Petersburg, Staryj Peterhof, Uljanovskaja 1. \em Serguei Naboko died on the 24th of December 2020.}

\author[Wood]{Ian Wood}
\address{Ian Wood, School of Mathematics, Statistics and Actuarial Sciences, Sibson Building,
 University of Kent, Canterbury, CT2 7FS, UK}
\email{i.wood@kent.ac.uk}

\begin{document}
\begin{abstract}\ian{
This paper is a contribution to the theory of functional models. In particular, it develops the so-called spectral form of the functional model where the selfadjoint dilation of the operator is represented as the operator of multiplication by an independent variable in
some auxiliary vector-valued function space. By using a Lagrange identity, in our version the connection between this auxiliary space and the original Hilbert space will be explicit. A simple example is provided.}
 \end{abstract}
\maketitle

\begin{center}
\textit{Dedicated to the memory of Professor B.S.Pavlov (1936-2016), outstanding mathematician and personality, who made great  contributions to the theory of functional models.}
\end{center}

\section{Introduction}

{The spectral and scattering properties of non-selfadjoint problems have become a subject of much
mathematical  and physical interest in recent years. Mathematically these problems pose a challenge, as apart from  exceptional cases, the well-developed methods used to examine the spectrum of  selfadjoint problems are not applicable. One of the tools to attack non-selfadjoint problems is functional models. }

Functional models were introduced by Sz.-Nagy and Foias (see \cite{SFBK10,Nik86} and references therein)  to analyse the structure 
of contractions and relations between an operator, its spectrum and its characteristic function, and independently by de Branges \cite{dB68}. 
These works built on the earlier papers  \cite{Liv46, Liv54} of Liv\v{s}ic for the triangular model. The ideas of Sz.-Nagy-Foias inspired great interest in the Soviet school. In particular, Pavlov \cite{Pav75} introduced a very useful symmetric version of the Sz.-Nagy-Foias model;
Nikolski and Vasyunin \cite{NikVas98} formulated a coordinate-free model; and Tikhonov \cite{Tik2004} re-developed the Sz.-Nagy-Foias model in the Nikolski-Vasyunin framework. As pointed out in \cite{NikVas89}, the coordinate-free model has the advantage of leaving the choice of spectral representation of the dilation to the user in the context of particular applications.

 Pavlov was always clear that his symmetric model should be used to solve real physical problems, and indeed his  work on quantum switches \cite{Pav02} and Naboko and Romanov's work on time asymptotics for the  Boltzmann operator \cite{NR01} have relied heavily on it. Pavlov himself developed a simplified version of his functional model for Schr\"{o}dinger operators \cite{Pav77}. Vasyunin \cite{Vas77} and Naboko \cite{Nab81} also introduced simplifications of Pavlov's original model; in particular, Naboko's model for additive perturbations is directly based 
 on Pavlov's model in \cite{Pav77}.
  
A drawback of many functional models is that their constructions require objects which may be difficult to describe explicitly, such as operator square roots, making it  hard to apply the results to specific examples.   In this context, Naboko's approach had at least two significant advantages: firstly, it gave explicit formulae for all expressions arising in the model, in terms of objects
which arise naturally in the description of the original operator (e.g. the imaginary part of the potential of a Schr\"{o}dinger operator); secondly, unlike approaches based on Cayley transformation to contractions,  it also allowed the study of non-dissipative operators.  Ryzhov's functional model for the case when the perturbation is only in the boundary conditions \cite{Ryz07} enjoys similar advantages, and inspired the work of Cherednichenko, Kiselev and Silva \cite{CKS19} on transmission problems for PDEs.

Our aim in recent work has been to develop a functional model for the case when the non-selfadjointness arises both in additive terms and in the boundary conditions. In a first paper, \cite{BMNW20}, we considered  a general maximally dissipative operator and   developed the so-called `translation form' of the functional model. We presented a construction of the selfadjoint dilation 
based on the Lagrange identity in the spirit of operator colligations \cite{Bro71,Bro78}. 
The flexibility of the choice of the $\Gamma$-operators in the Lagrange identity means that these can be chosen so that expressions arising in the dilation are given explicitly in terms of physical parameters (coefficients, boundary conditions and Titchmarsh-Weyl $M$-function) of the maximally dissipative operator. The presentation of such explicit expressions for the spectral
form of the functional model is arguably the main contribution of the present paper.

In the spectral form of the functional model, the dilation is very simple, being the operator of multiplication by an independent variable in
some auxiliary vector-valued function space; in our version the connection between this auxiliary space and the original Hilbert space will be explicit (Theorem \ref{thm4.5}).  Using the operator colligation setting for our problem, we also obtain an explicit expression for the completely non-selfadjoint part of the operator (Theorem \ref{thm:langer}) and an operator analytic proof of the famous result by Sz.-Nagy-Foias on the pure absolute continuity of the spectrum of the minimal selfadjoint dilation (Theorem \ref{thm:Foias}). In the final section of the paper, we consider an example of a limit circle Sturm-Liouville operator.

Throughout the paper we use the following notation: For a complex number $z\in\C$, let $\Im z$ denote its imaginary part and 
$\C_+=\{z\in\C:\Im z>0\}$, $\C_- =\{z\in\C:\Im z<0\}$. The positive half-line will be denoted by $\R_+$. For an operator $A$ in a Hilbert space $H$, we denote its range by $\Ran A$, its kernel by $\ker(A)$, its adjoint by $A^*$ and its spectrum and resolvent set by $\sigma(A)$ and $\rho(A)$, respectively. {The inner product on $H$ will be linear in the first component.} The set of bounded linear operators in $H$ is denoted $B(H)$.  The  Lebesgue space of square-integrable functions  on the half-line is $L_2(\R_+)$, while $H^s(\R_+)$ denotes the usual Sobolev space of order $s$;  $H^s_0(\R_+)$ denotes the closure in $H^s$-norm of the smooth, compactly supported functions on the half-line.

\section{Preliminaries}
This section reviews some classical results on dissipative operators - for more on the subject, we refer the reader to \cite{HP57,LP67,SFBK10} - and results from our previous paper \cite{BMNW20}, where most of the proofs can be found. We start with some basic definitions.

\begin{definition}
A densely defined linear operator $A$ with domain $D(A)$ in a Hilbert space $H$ is called \textit{dissipative} if $\Im\llangle Ah,h\rrangle \geq 0$ for all $h\in D(A)$. $A$ is called \textit{anti-dissipative} if $(-A)$ is dissipative.
Dissipative operators which have no non-trivial dissipative extensions are called \textit{maximally dissipative operators} (MDO).
\end{definition}

The Cayley transform, an operator version of the M\"{o}bius transform, defined by
\be\label{Cayley} T=I-2i(A+i)^{-1} = (A-iI)(A+iI)^{-1}\ee
is a bijective map between the class of MDOs and contractions that do not have $1$ as an eigenvalue. Thus many results for MDOs can be obtained from studying contractions, and vice versa.

We next introduce reducing subspaces and the concept of complete non-selfadjointness.
\begin{definition}
Let $A$ be an operator on a Hilbert space $H$, $H_1\subseteq H$ a subspace
and $P_{H_1}$ the orthogonal projection of $H$ onto $H_1$.
The subspace $H_1$
 is \textit{invariant} with respect to $A$ if $P_{H_1}D(A)\subseteq D(A)$ and $AP_{H_1}h\in H_1$ for all $h\in D(A)$.
It is a \textit{reducing subspace} for $A$ if both $H_1$ and $H\ominus H_1$ are invariant with respect to $A$.
\end{definition}

\begin{definition} Let $A$ be an MDO.  $A$ is \textit{completely non-selfadjoint (cns)} if there exists no reducing subspace $H_1\subseteq H$ such that $A\vert_{H_1}$ is selfadjoint.
\end{definition}

{The Langer decomposition \cite{BMNW20,Lan61,Nab81}
gives an explicit formula for the completely non-selfadjoint part of the operator. In the case of relatively bounded imaginary part the formula is simple. For more general situations it involves operators which are regularisations of the (possibly non-existing) imaginary part of the operator. In our setting, we will determine a more explicit formula for the completely non-selfadjoint part of an MDO in Theorem \ref{thm:langer}.}

\begin{proposition}[Sz.-Nagy]\label{dilation}
 For any MDO $A$ on a Hilbert space $H$ there exists a selfadjoint operator $\cL$ on a Hilbert space $\cH\supseteq H$ such that
 \beq
e^{it A}=P_H e^{it\cL}\vert_H,\ t\geq 0 \quad \hbox{ or equivalently } \quad (A-\lambda)^{-1}=P_H (\cL-\lambda)^{-1}\vert_H,\quad \lambda\in\C_-.
\eeq
Moreover, $(A^*-\lambda)^{-1}=P_H (\cL-\lambda)^{-1}\vert_H$ for $\lambda\in\C_+$.
The operator $\cL$ is called a \textit{selfadjoint dilation} of $A$.
\end{proposition}

{The selfadjoint dilation is a very useful tool in studying an MDO $A$. By decomposing $A$ into its selfadjoint and completely non-selfadjoint parts, it is sufficient to construct a selfadjoint dilation of the completely non-selfadjoint part to obtain a selfadjoint dilation for $A$. The next lemma, whose proof illustrates the use of the selfadjoint dilation,  will be needed later on.}

\begin{lemma} \label{lemstrong}
Let   $A$ be a maximally dissipative  operator in Hilbert space. Then for any $k\in \Rr $
$$ i \tau (A+k+ i \tau)^{-1} \overset{s}{ \to I}$$
in the  strong  operator  topology as $\tau\to +\infty$.
\end{lemma}

\begin{proof}
Since $(A+k)$  is also a maximally dissipative operator, the scalar  operator $kI$ can be absorbed in $A$. So without loss of generality $k=0$. Introducing the selfadjoint dilation $\mathcal L$ on Hilbert space $\mathcal H \supset H$ such that
$$(A+\lambda)^{-1}=P_H (\mathcal L+\lambda)^{-1}|_H   
 \quad\hbox{ for all }\  \lambda \in \ \mathbb C_+,
$$
where $P_H$ is the orthogonal projection of $\mathcal H$ onto $H$, we have
$$(i\tau) (A+ i\tau)^{-1}  = P_H (i \tau)(\mathcal L + i \tau)^{-1}  \quad\hbox{ for all }\  \tau>0.
$$
Therefore, for any $h\in H$
\begin{eqnarray*}
\norm{ i \tau (A+i\tau)^{-1}h-h}^2_H &=&
 \norm{ P_H[  i \tau (\mathcal L + i \tau)^{ -1}-I]h
}_H^2 \\
&\leq&\norm{   (\mathcal L
(\mathcal L    + i\tau)^{-1} h}_{\mathcal H}^2=\int_\Rr \left| \dfrac{t}{t+i\tau}\right|^2 d (E_t h,h)_\mathcal{H}
\end{eqnarray*}
with  $E_t$   the spectral resolution of $\mathcal L $  on $\mathcal H$.  Using the Lebesgue dominated convergence theorem and the trivial facts that
$|\dfrac t{t+i \tau}|^2 \leq 1$  and $ \dfrac t{t+i\tau} \to 0 $ as $ \tau\to \infty$,   for all  $t \in \Rr$,
we have $\norm  {  i \tau (A+i\tau)^{-1}h-h }_H\to 0$ as $\tau \to \infty$.
\end{proof}

We now discuss  an abstract framework for a  maximally dissipative operator and its anti-dissipative adjoint which allows us to introduce $\Gamma$-operators associated with the imaginary part of the operator $A$. For the case of bounded operators this goes back to the work of the Odessa school on operator colligations \cite{Bro71}, see also \cite{Str60}. The following is \cite[Lemma 3.1]{BMNW20}.

\begin{lemma}\label{lem:Lagrange}
 Let $A$ be a maximally dissipative operator on a Hilbert space $H$. Then there exists a Hilbert space $E$ and an operator $\Gamma:D(A)\to E$ which is bounded in the graph norm of $A$, has dense range in $E$ and such that for all $u,v\in D(A)$ we have
 \be \label{eq:Lagrange} \llangle Au,v\rrangle_H - \llangle u, Av\rrangle_H = i \llangle \Gamma u,\Gamma v\rrangle_{E}. \ee
 Similarly, there exists a Hilbert space $E_*$ and an operator $\Gamma_*:D(A^*)\to E_*$ which is bounded in the graph norm, has dense range in $E_*$ and such that for all $u,v\in D(A^*)$ we have
 \be \label{eq:LagrangeStar}\llangle A^*u,v\rrangle_H - \llangle u, A^*v\rrangle_H = -i \llangle \Gamma_* u,\Gamma_* v\rrangle_{E_*}.
 \ee
 \end{lemma}
We note that, in general, the dimensions of $E$ and $E_*$ need not coincide. The operator $\Gamma$ is determined up to unitary transformations, see \cite[Lemma 3.3]{BMNW20} and \cite{Str60}. {In particular, choosing $\Gamma = Q(A+i)^{-1}$ and $\Gamma_* = Q_*(A^*-i)^{-1}$, where $Q = (I-T^*T)^{1/2}$, $Q_* = (I-TT^*)^{1/2}$, and $T = (A-i)(A+i)^{-1}$ is
the Cayley transform of $A$, gives a mapping between the results presented here and those in \cite{NikVas98}. However in many concrete applications - see, e.g., Section \ref{section:5} below - 
(\ref{eq:Lagrange}) and (\ref{eq:LagrangeStar}) effectively reduce to integrations by parts, with much simpler canonical choices for $\Gamma$ and $\Gamma_*$.}

We require two abstract Green identities \cite[Lemma 3.5]{BMNW20}.

\begin{lem}\label{lem:Green} For $\lambda\in\C_+$ and $\mu\in\C_-$ we have
\be (A+\lambda)^{-1}-(A^*+\mu)^{-1} +(\lambda-\mu)(A^*+\mu)^{-1}(A+\lambda)^{-1} =  
 -i(\Gamma(A+\overline{\mu})^{-1})^*(\Gamma(A+\lambda)^{-1}) \label{Green1} \ee
and
\be
(A+\lambda)^{-1}-(A^*+\mu)^{-1} +(\lambda-\mu)(A+\lambda)^{-1} (A^*+\mu)^{-1}=  
 -i(\Gamma_*(A^*+\overline{\lambda})^{-1})^*(\Gamma_*(A^*+\mu)^{-1})
\label{Green2}
\ee
\end{lem}

\mm{A key ingredient in all functional models is a {\em characteristic function}, see \cite{Liv73}. 
We next introduce the characteristic function which we first presented in \cite[Corollary 4.2 \& Lemma 4.3]{BMNW20}.}

\begin{lemma} \label{corollary:2.2} 
Let  $z\in \C_+$. There exists a unique contraction $S(z):E\to E_*$, analytic in the upper half-plane, such that
\be S(z)\Gamma u = \Gamma_*(A^*-z)^{-1}(A-z)u \hbox{ for all } u\in D(A).\label{eq:S} \ee
Correspondingly, for $z\in\C_-$ there exists a contraction $S_*(z):E_*\to E$, analytic in the lower half-plane, such that
\be S_*(z)\Gamma_* u = \Gamma (A-z)^{-1}(A^*-z)u. \label{Sstar} \ee
\end{lemma}
{The characteristic function $S(z)$ can be extended on $\Ran(\Gamma)$ by \eqref{eq:S} to all $z\in\rho(A^*)$ and  $S_*(z)$ can be extended on $\Ran(\Gamma_*)$ by \eqref{Sstar} to all $z\in\rho(A)$ (\cite[Lemma 4.4]{BMNW20}).
The operator-valued function $S(\cdot)$, defined for $z\in\rho(A^*)$  by \eqref{eq:S} on $\Ran(\Gamma)$ and extended to $E$ by continuity is called the \textit{\v{S}traus characteristic function} of the operator $A$.}

{Finally, we gather some useful facts about the characteristic function in a lemma. The proofs can be found in \cite[Section 4]{BMNW20}.
\begin{lemma}\label{lemma:adj}
\begin{enumerate}
	\item For $\mu,\mut\in\rho(A^*)$, we have the following identity:
\be\label{eq:Sdiff} S(\mu)-S(\mut)= i(\mu-\mut) \left(\Gamma_* (A^*-\mu)^{-1}\right)\left(\Gamma(A-\mutb)^{-1}\right)^* \hbox{ on } E.\ee
\item $S(z)=S_*^*(\zb)$ for $z \in \rho(A^*)$.
\item $S(z)S_*(z)=I_{E_*}$ and $S_*(z)S(z)=I_E$ whenever  $z \in \rho(A)\cap\rho(A^*)$.
\item $S(z)$ is unitary for $z\in\R\cap\rho(A)$.
\item If $\sigma(A)$ does not cover the whole upper half plane (or, equivalently, if $\rho(A)\cap\rho(A^*)\not=\emptyset$), then $\dim E=\dim E_*$.
\item  For  $w,z \in \C_+$, we have
\be\label{eq:SStarS1}\frac{1}{\bar{w}-z}\left(I_E-S^*(w)S(z)\right)=i \left(\Gamma(A-\bar{w})^{-1}\right) \left(\Gamma(A-\bar{z})^{-1}\right)^*\ee
and for $w,z \in \C_-$, we have
\be\label{eq:SStarS2}\frac{1}{\bar{w}-z}\left(I_{E_*}-S_*^*(w)S_*(z)\right)= - i \left(\Gamma_*(A^*-\bar{w})^{-1}\right) \left(\Gamma_*(A^*-\bar{z})^{-1}\right)^*.\ee
\item For any $u\in H$, $\mu,z\in\C_-$ we have
\be\label{eq:Sgammastar}
\left(\Gamma_*(A^*-\bar{\mu})^{-1}\right)^*S(\bar{z}) = \left[I-(\bar{z}-\mu)(A-\mu)^{-1}\right]\left(\Gamma(A-z)^{-1}\right)^*
\ee
and
\be\label{eq:Stildegammastar}
\left(\Gamma(A-\mu)^{-1}\right)^*S_*(z) = \left[I-(z-\bar{\mu})(A^*-\bar{\mu})^{-1}\right]\left(\Gamma_*(A^*-\bar{z})^{-1}\right)^*.
\ee
\end{enumerate}
\end{lemma}
}

{\section{Absolute continuity of the spectrum}}
    
    We start with the following important fact to be used frequently over the paper. It is a generalisation of \cite[Theorem 1]{Nab81}  for the case of  general maximally dissipative  operators  and is  in the spirit of operator colligations.
\begin{theorem} \label{thm:2.1}
Let $A$ be a maximally dissipative positive operator in $H$. Then
$$\sup_{\eps >0} \int_\Rr
\norm{ \Gamma(A-k+ i \eps )^{-1} u}_E^2 dk \leq 2 \pi \norm{u}^2
$$
and 
$$\sup_{\eps >0} \int_\Rr
\norm{ \Gamma_*(A^*-k- i \eps )^{-1} u}_{E_*}^2 dk \leq 2 \pi \norm{u}^2.
$$
In   other words for any vector $u\in H$ the vector valued function $\Gamma (A-z)^{-1}u \in E$ for $\Im z <0$ belongs to the vector-valued Hardy class $H^-_2(E)$ (see, e.g.~\cite{SFBK10}) of $E$-valued analytic functions in the lower half plane.
Similarly $\Gamma_* (A^*-z)^{-1} u \in H^+_2(E_*)$, the Hardy class  of $E_*$-valued analytic functions on        $\mathbb C_+$.
\end{theorem}

\begin{proof}
According to  
 (\ref{eq:Lagrange})
\begin{eqnarray*}
&&\int_\Rr \norm{  (\Gamma (A- k + i \eps)^{-1}u}_E^2 dk
=  \int_\Rr (  \Gamma (A- k + i \eps)^{-1}u, \Gamma (A- k + i \eps)^{-1}u)_E  dk\\
&& = \dfrac1i  \int_\Rr [(A(A-k+ i \eps)^{-1}u, (A-k+i \eps)^{-1} u)_H  - ( (A-k+i\eps)^{-1}u, A (A-k+i\eps)^{-1}u)_H ]dk\\
&&= \dfrac1i\int_\Rr \{  (u+(k-i\eps) (A-k+i \eps)^{-1}u,  (A-k+i \eps)^{-1}u)_H -( (A-k+i\eps)^{-1} u, u+ (k-i\eps) (A-k+i \eps)^{-1}))_H \} dk\\
&&=\dfrac1i \int_\Rr \{ (u, (A-k+i\eps)^{-1} u )_H-(  (A-k+i\eps)^{-1}u,u)_H -2i\eps  (  (A-k+i\eps)^{-1}u, (A-k+i\eps)^{-1})_H\}dk \\
&&= \int_\Rr \{ 2 \Im   (  u, (A-k+i \eps)^{-1}u)_H -2 \eps
\norm{ (A-k+i\eps)^{-1}u}_H^2\} dk
\leq  2 \int_\Rr \Im (u, (A-k+i \eps)^{-1}u)_H dk .
\end{eqnarray*}
Since, by Proposition \ref{dilation},
$(u, (A-k+i\eps)^{-1}u)_H=(u, P_H(\mathcal L-k+i\eps)^{-1}u)_\mathcal H
$  

for the selfadjoint dilation $\mathcal L$ of $A$ in the Hilbert space $\mathcal H 
 \supseteq H$,
 we have 
\begin{eqnarray*}
\int_\Rr \norm{ \Gamma (A-k+i \eps)^{-1} u }_E^2 dk &\leq& 2 \int_\Rr \Im ( ( \mathcal L-k-i\eps)^{-1}u,u)_{\mathcal H}dk \\
&=& 2 \int_\Rr \Im  \int_\Rr \dfrac1{\lambda-k-i\eps}  d (E_\lambda u,u)_{\mathcal H}dk=
2\int_\Rr   dk \int_\Rr  \dfrac \eps  {  (\lambda -k)^2 +\eps ^2}  d ( E_\lambda u,u)_{\mathcal H},
\end{eqnarray*}
where $E_\lambda$ is the spectral resolution of $\mathcal L$.
Due to positivity of the function and the measure one may use Fubini's Theorem  rewriting the last double integral as
$$\int_\Rr d (E_\lambda u, u)_{\mathcal H} \left (  \int_\Rr dk \dfrac{2\eps }{  (\lambda-k)^2 +\eps^2}\right )=\int_\Rr d(E_\lambda u,u)_{\mathcal H} (2 \pi)=2\pi \norm{u}^2_{\mathcal H}=2 \pi\norm{u}^2_{ H}$$
since in the integral over variables $k$   does not depend on $\lambda$ (by the shift of variables $k\to k-\lambda$) and is equal to $2 \pi$.

The second inequality in the Theorem \ref{thm:2.1} admits exactly the same proof.
\end{proof}

{The proof of the theorem includes two identities:}
\begin{corollary}
We have
$$ \mathrm{(i)}
\int_\Rr \norm{ \Gamma  (A-k+i\eps)^ {-1}u}_E^2 dk 
= 2 \pi\norm{u}^{2}_H- 2 \eps \int_\Rr \norm{  (A-k+i\eps)^{-1}u}_H^2 dk \hbox{ for all } \eps >0$$  and  
$$ \mathrm{(ii)}
\int_\Rr 
\norm{ \Gamma_* (A^*-k-i\eps)^{-1}u}_{E_*} dk = 2\pi\norm{u}^2_H
 - 2 \eps \int_\Rr \norm{  (A^*-k-i\eps)^{-1}u}^2 _H
 dk \hbox{    for all } \eps >0.$$
 In particular, we have useful bounds for an arbitrary  maximally dissipative operator   $A$:
 $$ \sup_{\eps >0} \int_\Rr \eps \norm { (A-k+i\eps)^{-1}u}_H^2 dk \leq \pi \norm{u} _H^2 $$
 and
 $$ \sup_{\eps >0} \int_\Rr \eps \norm { (A^*-k-i\eps)^{-1}u}_H^2 dk \leq \pi \norm{u}_H ^2. $$\end{corollary}
 
{The next result is a formulation of the Langer decomposition \cite{Lan61} in a form which will be convenient for our later applications.
Equivalent representations of the completely non-selfadjoint subspace may be deduced from expressions from completely-nonunitary
parts of contractions, see e.g. \cite[Section 6]{NikVas98}. }

 \begin{thm}  \label{thm:langer}
 The reducing subspace of the maximally dissipative operator $A$ corresponding to its completely non-selfadjoint part in the Langer decomposition is
 \be\label{eq:cns}
H_{cns}=\bigvee   \left\{ \bigvee_{\Im \lambda >0}  ( \Gamma (A+\lambda)^{-1})^*E, \bigvee_{\Im \mu <0} (\Gamma_* (A^*+\mu)^{-1})^*E_*)\right\}
\ee
 and its selfadjoint part $H_{sa}: = H\ominus H_{cns}$.
 \end{thm}
 \begin{proof}
  Denote  the right hand side of \eqref{eq:cns} by ${\mathcal M}_{A}$.
 {Our proof consists of two inclusions, identifying $H_{sa}$  with the orthogonal complement  of ${\mathcal M}_A$.}

 We first show:
  $ H_{sa} \subseteq {\mathcal M}_{A}^\perp$.
 Let   $h\in H_{sa}$  
 then using (\ref{Green1}), the first abstract Greens function identity, one gets for  $\lambda=\overline {\mu} \in    \mathbb C_+$
 $$  -i (\Gamma (A+\lambda)^{-1})^* (\Gamma (A+\lambda)^{-1})h=(A+\lambda)^{-1} h -(A^*+\overline \lambda)^{-1} h + (\lambda-\overline \lambda)(A^*+\overline \lambda)^{-1} (A+\lambda)^{-1} h
 =0 
 $$
 by the Hilbert identity for the selfadjoint operator  $A|_{H_{sa}}=A^*|_{H_{{sa}}}$.
Therefore 
  $$\norm{ \Gamma (A+\lambda)^{-1}h}^2_E = \\
 ( ( \Gamma (A+\lambda)^{-1})^* ( \Gamma (A+\lambda))^{-1} h,h)_H=0,$$  i.e.
 $\Gamma(A+\lambda)^{-1}h
\equiv 0$   for all $\lambda \in \ \mathbb C_+$.

 Similarly, using the second identity of (\ref{Green2}), we have
$$\Gamma_*(A^*+\mu)^{-1} h\equiv 0   $$
for $\mu \in \ \mathbb C_-$. 
 The last two conditions on $h$    mean that $h\perp \mathcal M_A$, or   $h\in \mathcal M^{\perp}_A$ proving the
  inclusion $H_{sa}\subseteq \mathcal {M_A}^\perp$.
  
It remains to show:
$H_{sa}\supseteq \mathcal {M_A}^\perp.$
Let the vector $h\neq 0$  and $h\in \mathcal M_A^\perp$, i.e.
$\Gamma (A+\lambda)^{-1}h=0$  and $\Gamma_* (A+\mu)^{-1}h=0$  for any $\lambda \in \ \mathbb C_+$  and $\mu \in \ \mathbb C_-$.   
Consider a reducing  subspace for $A$,  $\mathcal \Eta_h$ in $H$ generated by the vector $h$:
$$ \eta_h=\bigvee \{ (\bigvee_{\Im \lambda >0} (A+\lambda)^{-1} h),    (\bigvee_{\Im \mu < 0} (A^*+\mu)^{-1} h)\}.$$
Its reducing property, which follows from the invariance with respect to both resolvents $(A+\tilde \lambda)^{-1}$  
  and  $(A^*+\tilde \mu)^{-1}$, $\tilde \lambda
  \in \C_+$  and    $\tilde \mu
  \in \C_-$ can be easily proved by using the Hilbert identity for resolvents of $A$ and $A^*$  respectively.
Indeed, for $\Im \tilde \lambda >0,  \tilde \lambda   \neq   \lambda;  \alpha, \beta \in  \mathbb C$
\begin{eqnarray*}
&&  (A+\tilde \lambda)^{-1}  (\alpha (A+\lambda)^{-1} h + \beta (A^*+\mu)^{-1} h)\\
&&= \alpha ( (A+\lambda)^{-1} -(A+\tilde \lambda )^ {-1} )  (\tilde  \lambda -\lambda)^{-1}h+\beta (A+\tilde \lambda)^{-1} (A^*+\mu)^{-1}h
\\
&&=\alpha ( (A+\lambda)^{-1}- (A+ \tilde \lambda)^{-1}) (\tilde \lambda-\lambda)^{-1}h +\beta (\tilde \lambda -\mu)^{-1}  [  (A^*+\mu)^{-1}-(A+\lambda)^{-1}]h
\end{eqnarray*}
 is again  a linear combination of vectors of the type
$(A+\lambda)^{-1} h $ and   $ (A^*+\mu)^{-1} h$.  Here we used formula (\ref{Green2}),   the second abstract Green function identity together with the fact that $\Gamma_* ( A^*+\mu)^{-1}h=0$, $\mu \in \ \mathbb C_-$.
Invariance with respect to the resolvent $(A^*+\tilde \mu)^{-1}$, $\tilde \mu \neq \mu$, $\tilde \mu \in \ \mathbb C_-$  has a similar proof  using the first abstract Green  function identity (\ref{Green1})   
and the fact that $\Gamma (A+\lambda)^{-1}h=0$.  {The exceptional cases $\tilde \lambda  = \lambda$ and $\tilde \mu  = \mu$ are achieved 
by the limit procedures $\tilde \lambda \to \lambda$ and $\tilde \mu\to \mu$ respectively.}

{By Lemma \ref{lemstrong}, we have $h=\lim_{ {\tau }\to +\infty} (A+i \tau)^{-1} (i \tau)h$. This implies that $h\in\eta_h$.
Therefore, it is sufficient to prove that the  reduced operator $A|_{\eta_h}$ is a selfadjoint operator in    $ \mathcal \eta_h $, since then $h\in\eta_h\subseteq H_{sa}$. }
Let us  consider the Cayley transform of $A|_{\eta_h}$
$$ T_h:=(A-i)(A+i)^{-1}|_{\eta_h}
=(I-2i (A+i)^{-1})|_{\eta_h}
$$
and show that $T_h$ is a unitary operator on $\eta_h$.  Indeed, noting that $ (T_h)^*=(I+2i (A^*-i)^{-1})|_{\eta_h}$, we see
\begin{eqnarray*}
(T_h)^*T_h&=&(I + 2 i (A^*-i)^{-1})(I-2i (A+i)^{-1}))  |_{\eta_h}\\
&=&(I + 2 i(A^*-i)^{-1} - 2 i (A+i)^{-1}-(2 i)^2 (A^*-i)^{-1}  (A+i)^{-1})  |_{\eta_h}.
\end{eqnarray*}
Again by the first abstract Green function formula   (\ref{Green1})
\begin{eqnarray*}
&& [ (A+\lambda)^{-1} - (A^*+\mu)^{-1}  + (\lambda-\mu) (A^*+\mu)^{-1} (A+\lambda)^{-1}]h\\
&&= -i (\Gamma (A+\overline \mu)^{-1})^* ) (\Gamma (A+\lambda)^{-1})h=0
\end{eqnarray*}
and similarly by the second abstract Green function formula  (\ref{Green2})
\begin{eqnarray*}
&& [(A+\lambda)^{-1}-(A^*+\mu)^{-1} + (\lambda-\mu)(A+\lambda)^{-1} (A^*+\mu)^{-1}]h  \\
&&   =  -i (\Gamma_* (A ^*+ \overline \lambda)^{-1} )^* (\Gamma_* (A^*+\mu)^{-1} ) h=0
\end{eqnarray*}
and therefore, comparing  the last two formulae, we have shown
\begin{equation}
(A+\lambda)^{-1}  (A^*+ \mu)^{-1}  h =  (A^*+\mu)^{-1} (A + \lambda)^{-1}h  \label{8}
\end{equation}
for any $\lambda \in \ \mathbb C_+$, $\mu\in \ \mathbb C_-$.
 {So 
 \begin{eqnarray*}
(T_h)^*T_h(A+\lambda)^{-1}h&=&  (I + 2 i (A^*-i)^{-1}-2 i (A+i)^{-1}-(2 i)^2 (A^*-i)^{-1}  (A+i)^{-1}) (A+\lambda)^{-1} h\\
&=&   [(A+\lambda)^{-1
}+ 2 i (A+\lambda)^{-1} (A^*-i)^{-1}-  \\
&& -2i (A+\lambda)^{-1}  (A+i)^{-1} -(2 i)^2 (A^*-i)^{-1} (A+i)^{-1}(A+\lambda)^{-1}]h\\
&=&(A+\lambda)^{-1} [  h + 2 i (A^*-i)^{-1} h - 2 i (A+i)^{-1} h]\\
&&-(2i)^2 (A^*-i)^{-1} \{  (A+i)^{-1} - (A+\lambda)^{-1} \} (\lambda-i)^{-1}h\\
&=&  (A+\lambda)^{-1}[  h + 2 i (A^*-i)^{-1} h- 2 i (A+i)^{-1}h]\\
&&-(2i)^2 \{ (A+i)^{-1} - (A+\lambda)^{-1} \} (\lambda-i)^{-1}(A^*-i)^{-1} h \\
&=&  (A+\lambda)^{-1}[ h + 2 i (A^*-i)^{-1}h-2i(A+i)^{-1}h]\\
&&-(2i)^2 (A+\lambda)^{-1}(A+i)^{-1} (A^*-i)^{-1}h = (A+\lambda)^{-1}h,
\end{eqnarray*}}
where in the last step the second abstract Green identity (\ref{Green2}) was used   and $h \in M_A^\perp$.
 Hence
$$(T_h)^*T_h(A+\lambda)^{-1}h =(A+\lambda)^{-1}h, \forall  \lambda \in \ \mathbb C_+$$
Similar calculations which are even simpler because we can use (\ref{8}) explicitly, show 
$$(T_h)^*T_h(A^*+\mu)^{-1}h=(A^*+\mu)^{-1} h,  \forall\mu \in \ \mathbb C_ -.$$
{
Since any element of $\eta_h$ is a limit of linear combinations of vectors $(A+\lambda)^{-1}h$, and $(A^*+\mu)^{-1}h$  and $T_h$ is a bounded operator we have proved the isometry of $T_h$: $(T_h)^*T_h=I$  on $\eta_h$.
The second identity, $T_h(T_h)^*=I$ on $\eta_h,$ admits a similar proof because all calculations above are symmetric with respect to both $A$  and $A^*$.}

{This proves that $A\vert_{\eta_h}$ is selfadjoint, so $\eta_h\subseteq H_{sa}$. Since $h\in \eta_h$, this shows the required inclusion
$\mathcal {M_A}^\perp\subseteq H_{sa}$, completing the proof.}
\end{proof}

The next theorem demonstrates one of the deepest  results of dilation theory. Its original proof, given in \cite{SFBK10}, is based on some ideas of a geometric nature. 
{Actually, this theorem was first proven for the case of contractions and their unitary dilations.  However the fact can be easily transfered to   the dissipative situation using the Cayley transform.
Below we suggest a new proof  based essentially on Theorem \ref{thm:2.1},  i.e. applying  operator analytic arguments.}

\begin{theorem} (B.Sz.-Nagy - C.Foias \cite{SFBK10}) \label{thm:Foias}
The minimal selfadjoint 
dilation   { of a   completely non-selfadjoint maximally dissipative operator has purely absolutely continuous spectrum
covering the whole real line.}
\end{theorem}
\begin{proof}

According to Theorem {\ref{thm:langer}} complete non-selfadjointness leads to the fact that
$$H=H_{cns}=\bigvee \{ \Span_{\Im \lambda >0} (\Gamma (A+\lambda)^{-1})^* E,\; \ \Span _ { \Im \mu<0 }(\Gamma_*(A^*+\mu)^{-1})^*E_* \}.$$

Consider two linear sets of test vectors generating $H$:
$$\mathcal L_1:=\Span_{\Im \lambda >0} (\Gamma (A+\lambda)^{-1})^*E; \;\mathcal L_2  :=\Span_{\Im \mu <0}(\Gamma_* (A^*+\mu)^{-1}E_*)
.$$
Theorem \ref{thm:2.1} shows  that for any $$u \in H,    g \in E,   g_*   \in E_*:$$  
$$(\Gamma (A+\lambda)^{-1}u,g)_E=  (u, ( (\Gamma (A+\lambda)^{-1})^* g)_H\in H_2^+
$$
and
$$(\Gamma_*(A^*+\mu)^{-1}u,g_*)_{E_*} 
\equiv (u, ( (\Gamma_* (A^*+\mu)^{-1} )^*g_*)_H\in H^-_2.$$
Introducing the auxiliary parameter
$\lambda_0 \in \ \mathbb C_+$ 
we have by the Hilbert identity 
\begin{eqnarray*}
&& (  \Gamma (A+\lambda_0)^{-1}   (A+\lambda)^{-1}u,g )_E =   (\Gamma(A+\lambda_0)^{-1}[(A+\lambda)^{-1}u],g)_E\\
&& = ([ (A+\lambda)^{-1}u], ( \Gamma(A+\lambda_0)^{-1})^*g)_H   \in H_2^+, \forall g\in E
\end{eqnarray*}
as a function of $\lambda \in \ \mathbb C_+$.
Similarly,
$$(\Gamma_*(A^*+\mu_0)^{-1} (A^*+\mu)^{-1}u,g_*)_{E_*} = ((A^*+\mu)^{-1}u, (\Gamma_* (A^*+\mu_0)^{-1})^*  g_*)_{H}\in H_2^{-}, \forall g_*\in E_*,
$$
$\Im \mu_0<0$, as a function of $\mu \in  \ \mathbb C_-$.
So introducing the minimal selfadjoint dilation $\mathcal L $ of    $A$ in $\mathcal H$ we proved, assuming without loss of generality $\mathcal H  \supset H$,
$$
( P_H(\mathcal   L+ \lambda)^{-1} u, (\Gamma (A+\lambda_0)^{-1})^*g )_\cH \in H_2^+, \forall g \in E, \forall u\in H, \forall \lambda_0\in \ \mathbb C_+
$$
as a function of $\lambda \in \ \mathbb C_+$    and
$$ ( P_H(\mathcal L+\mu)^{-1} u, ( \Gamma_*(A^*+\mu_0)^{-1})^* g_*)_\cH \in H_2^{-}$$
as a function of $\mu \in \ \mathbb C_-$  for arbitrary $u\in H, \mu_0 \in \ \mathbb C_-$ and $g_*\in E_*.$
Taking  linear combinations of the test vectors one gets
$$(P_H(\mathcal L+ \lambda)^{-1} u, \phi)_\cH
{=  ((\mathcal L+ \lambda)^{-1} u, \phi)_\cH}
\in H_2^+,\forall u \in H, \phi \in \mathcal  L_1$$
and
$$(P_H(\mathcal L + \mu)^{-1} u, \psi)_\cH
 {=((\mathcal L + \mu)^{-1} u, \psi)_\cH}
\in H_2^-,\forall u \in H, \psi \in \mathcal  L_2.$$
Introducing the spectral resolution $E_t$ of $\mathcal L$ in $\mathcal H \supset H$  we can rewrite our conditions as follows:
$$ \int_\Rr \dfrac1{t+\lambda } d (E_t u, \phi )_\mathcal  H \in H^+_2 ,\forall u \in H, \; \phi\in \mathcal L_1 $$
and
$$ \int_\Rr \dfrac1{t+\mu} d (E_t u, \psi )_\mathcal  H \in H^-_2 ,\forall u \in H, \psi \in \mathcal L_2.
$$
By the standard representation theorem for Hardy classes  there exist two $L_2(\Rr)$ scalar functions  $f_\pm$ (depending on $u$, $\phi$, $\psi$)
such that
$$\int_\Rr \dfrac1{t+\lambda} d (E_t u, \phi)_\mathcal H=\dfrac1{2\pi}  \int_\Rr  \dfrac{ f_+(t) }{t+\lambda} dt, \forall \lambda\in \ \mathbb C_+
$$
and
$$\int_\Rr \dfrac1{t+\mu} d (E_t u, \psi)_\mathcal H=-\dfrac1{2\pi}  \int_\Rr  \dfrac{ f_-(t) }{t+\mu} dt, \forall \mu\in \ \mathbb C_-.
$$
Therefore
$$\int_\Rr \dfrac1{t+\lambda} [ d (E_t u, \phi)_\mathcal H-\dfrac1{2\pi i } f_+(t)dt]\equiv 0, \lambda\in \ \mathbb C_+
$$
and
$$\int_\Rr \dfrac1{t+\lambda} [ d (E_t u, \psi)_\mathcal H  +\dfrac1{2\pi i } f_-(t)dt]\equiv 0, \lambda\in \ \mathbb C_+.
$$
The  {F. and M. Riesz} Theorem \cite{Koosis}  implies that the complex measures
$\{ d(E_t u, \phi)_\mathcal H-\dfrac1{ 2 \pi i} f_+(t) dt \} $ and  $\{ d(E_t u, \psi)_\mathcal H+\dfrac1{ 2 \pi i} f_-(t) dt \} $
are both absolutely continuous, i.e.
$ d (E_t u, \phi)_\mathcal H$  and $ d (E_t u, \psi)_\mathcal H$   are both  absolutely continuous for any $\phi \in \mathcal L _1$  and $\psi \in \mathcal L_2$ respectively.  Summing up,  $d (E_t u, v)_\mathcal H$ is an absolutely  continuous measure for any $u\in H$  and  any vector $v\in \mathcal L_1+ \mathcal L_2 =: \mathcal L_3.$
 As a final step let us consider the expression
\begin{eqnarray*}
&&
d (E_t (\mathcal L + \overline \gamma )^{-1}(\mathcal L + \sigma)^{-1} u, v )_\mathcal H=  \dfrac1{(t+\sigma)(t+\overline \gamma)}  d(E_t u,v)_\mathcal H
\\
 &&= d ( E_t (\mathcal L + \sigma)^{-1}u, (\mathcal L + \gamma)^{-1}v )_\mathcal H
\end{eqnarray*}
which is an absolutely continuous complex measure  for any values of non-real parameters $\sigma$ and $\gamma$.
Since  $\mathcal L$ is a minimal selfadjoint dilation of $A$ the  $\Span \{ (\mathcal L+\sigma)^{-1} u :  u \in H \;{\rm and \;} \sigma \in \ \mathbb C\backslash \Rr \}  $
is dense in $\mathcal H$. The same is   true   for   $\Span \{ (\mathcal L+\gamma)^{-1}v :   v \in \mathcal L_3 \;{\rm and \;} \gamma \in \mathbb C \backslash \Rr \}  $.
So the measure
 {$d(E_t x,y)_\mathcal H$ is pure absolutely continuous for dense sets of vectors $x$  and $y$ in $\mathcal H$.}  Therefore the minimal dilation $\mathcal L$ has pure absolutely continuous spectrum.

The spectrum of $ \mathcal L$ has to cover the whole real line. 
Indeed, assume that its spectrum has a gap, which includes  an interval  $I$.  Then the formula
in Proposition \ref{dilation}     connects the resolvents of the dilation and of the  operators $A $ and $A^*$, showing   that all three resolvents admit   analytic continuations from the appropriate  complex half-plane  to the interval $I$ and therefore coincide there. Hence $A=A^*$.
\end{proof}

\section{Spectral form of the selfadjoint dilation of a maximally dissipative operator}

By von Neumann's   general theory of selfadjoint operators in Hilbert spaces (see, e.g.~\cite{BS87}), the abstract Hilbert space can be replaced by the space of $L_2$-summable functions $\eta(k)$ with  $k\in \Rr$, taking values in auxiliary Hilbert spaces, such that the minimal dilation of the completely non-selfadjoint maximally dissipative operator $A$ is represented in that space by a multiplication operator by the independent variable $k\in \Rr$.

In {this section}  we will make this procedure explicit, transforming the  translational  form \cite{Nab81}  of $\mathcal L$ into its explicit spectral form. Another important feature of the approach is that both  transforms to translational form and later to the spectral form will be performed in an explicit way using the operator colligations method, as well as the Strauss characteristic function.  {The spectral form will be presented in B.~Pavlov's version  \cite{Pav75}}, symmetric with respect to both incoming and outgoing subspaces \cite{LP67,Pav77}. That form, in our opinion, has some advantages compared to  both the standard B.~Sz-Nagy-C.~Foias form of the selfadjoint (unitary) dilation and  also  the   L.~de Branges \cite{dB68} {form}.

We first recall ({in the notation of} \cite[Section 5]{BMNW20}) the explicit construction procedure of the selfadjoint dilation of a maximally dissipative operator $A$ in  a Hilbert space 
$H$. 

{The linear set defined next will be the domain of the selfadjoint dilation of $A$ in the so-called translation form.}
 \begin{definition} 
  Let  $\mu \in \mathbb C_-$ and $\lambda \in \mathbb C_+$ {and consider the Hilbert space $\mathcal{H}_{tr} := L_2(\Rr_-, E_*) \oplus H\oplus L_2 (\Rr_+,E)$.
   Define the linear subset  $\mathcal {D ({L}}_{tr })$ by 
	}

\begin{eqnarray}\label{def:L}
 D(\cL_{tr}) & = & \left\{ U=\U\in\cH_{tr} \; : \; u\in H,\; v_+\in H^1(\R_+,E),\; v_-\in H^1(\R_-,E_*), \; \right. \\
 &  \mbox{{\bf (I)}} & \hspace{1cm} \begin{array}{l}u+(\Gamma_*(A^*+\mu)^{-1})^*v_{-}(0)\in D(A)  \hbox{ and }  \\
                                                                          v_+(0)=S^*(-\mu)v_-(0)+i\Gamma\left(u+(\Gamma_*(A^*+\mu)^{-1})^*v_{-}(0)\right)\end{array},
\nonumber \\ 
    & \mbox{{\bf (II)}} & \hspace{1cm} \left. \begin{array}{l}u+(\Gamma(A+\lambda)^{-1})^*v_{+}(0)\in D(A^*) \hbox{ and }\\
	 v_-(0)=S(-\bar{\lambda})v_+(0)-i\Gamma_*\left(  u+(\Gamma(A+\lambda)^{-1})^*v_{+}(0) \right)\end{array} \quad\vphantom{\U} \hspace{-5mm}\quad \right\}.  \nonumber 
\end{eqnarray}
\end{definition}

Here, $H^1 (\Rr_+,E)$   and $H^1(\Rr_-, E_*)$ are the Sobolev spaces of vector-valued functions on $\Rr_+$ and $\Rr_-$ respectively, taking values on the auxiliary Hilbert spaces $E$, or   correspondingly   $E_*$.  The norm of the spaces is given by 
$$
\int_0^\infty  (  \norm{ v_+'(\xi)}^2_E +\norm{v_+(\xi)}^2_{E}) d \xi=:\norm{v_+}^2_{H^1(\Rr_+, E)}
$$
and
$$
\int_{-\infty}^0  (  \norm{ v_-'(\xi)}^2_{{E_*}} +\norm{v_-(\xi)}^2_{E_*}) d \xi=:\norm{v_-}^2_{H^1(\Rr_-, E_*)}.
$$
{
It follows that both  $v_+(0) := \lim_{\xi\to 0^+}  v_+(\xi)
$  and
$v_-(0):= \lim_{\xi \to 0^-} v_+(\xi)$ { are well-defined 
in   the $E$  and   $E_*$ topologies,  respectively.}}
\begin{Remark}  
{Note that  whenever $u\in H$ and $v_-(0)\in E_*$  are such that $[u + (\Gamma_*(A^*+\mu_0)^{-1})^*v_-(0))]\in \mathcal D (A)$ for some $\mu_0\in  \C_-$, then
$
[u+(\Gamma_*(A^*+\mu)^{ -1})^*
v_-(0)]\in \mathcal
D(A)
$
for all $\mu\in \C_-$ (see \cite[Lemma 5.2]{BMNW20}). Similarly, the  second condition
$[u+(\Gamma(A+\lambda)^{-1})^*u_+(0)]\in \mathcal D(A^*)$  does not depend on the choice of $\lambda \in  \C_+$.   Although,   we denoted  vectors in the description of $D(\cL_{tr})$ from $E_*$ and $E$ in the form $v_-(0)$ and $v_+(0)$ to be suggestive of their role in applications, both vectors can be chosen arbitrarily in the respective spaces.}

We should mention that in (\ref{def:L}), conditions {\bf (I)} and {\bf (II)} are equivalent, {see e.g.} \cite[Lemma 5.4]{BMNW20}.
Finally, we see that there are only four
free parameters in the domain of $\mathcal L_{tr}$. 
These can be chosen as
\begin{enumerate}
\item
the vector   $v_+(0)\in E$
\item
a vector $h\in \mathcal D(A^*)$ such that $u:=h-(\Gamma (A+\lambda)^{-1})^* v_+(0)$,
one can take $\lambda=i$ here for example;
\item
two vector-valued functions {$w_+ \in H^1 (\Rr_+, E)$ and $w_-\in H^1 (\Rr_-, E_*)$ with $w_+(0)=0, w_-(0)=0$}.
\end{enumerate}

{Indeed, according to the equivalence of the conditions {{\bf (I)}   and {\bf (II)} of (\ref{def:L})}, one can choose  a vector $(v_-, u, v_+)$  from $\mathcal{ D(L}_{tr})$ such that
\begin{enumerate}
\item
$v_+(\xi):=  w_+(\xi)+v_+(0) e^{-\xi}, \; \xi\geq 0$
\item
$v_-(\xi):=  w_-(\xi)+ (S(-\overline \lambda) v_+(0) -i \Gamma _* h)e^{\xi}, \; \xi\leq 0$
\item
$u:=h-(\Gamma (A+\lambda)^{-1})^* v_+(0)$
\end{enumerate}
for any  fixed $\lambda \in  \C_+$, say $\lambda=i$.}
 \end{Remark}

In order to introduce the formula for the  dilation $\mathcal {L}_{tr}$ we need the following  definition.
\begin{Definition} \label{def:T} 
Let $\mu \in  \C_-$  and $\lambda \in  \C_+$  be fixed.   For any vector $U= (v_-,u,v_+)   \in \mathcal { D(L}_{ tr})$  define two operators $T$ and $T_*: \mathcal { D(L}_{ tr})\to H$ by
\be\label{eq:T} TU:= A^*(u+{(\Gamma (A+\lambda)^{-1}})^{*}v_+(0))+\overline \lambda (\Gamma (A+\lambda)^{-1})^*v_+(0)\ee
  and
\be\label{eq:Tstar} T_*U= A(u+{(\Gamma_*(A^*+\mu)^{-1})}^*v_-(0))+\overline \mu (\Gamma _*(A^*+\mu)^{-1})^*v_-(0)\ee
\end{Definition}

{We note that $T\equiv T_*$ on $\mathcal { D(L}_{ tr})$ and therefore both are independent of $\lambda$ and $\mu$, see \cite[Lemma 5.7  \& Corollary 5.8]{BMNW20}.}

Now the selfadjoint dilation of  the maximally dissipative operator $A$ in $H$ can be defined as follows:
\begin{Definition}\label{def:op} {For any vector $U= (v_-,u,v_+)   \in \mathcal { D(L}_{ tr})\subset L_2(\Rr_-, E_*)\oplus H\oplus L_2(\Rr_+,E)$, set}
$$  {  \mathcal L_{tr} U} \equiv \mathcal L _{tr} \left (\begin{array}{c} {v}_-\\u\\{v}_+\end{array} \right ) = \left (\begin{array}{c} iv'_-\\TU\\{iv'}_+\end{array} \right ).
$$
\end{Definition}

{Therefore, the operator $\mathcal L_{tr}$  acts,   both in the incoming  channel  $$D_-=(L_2(\Rr_-, E_*), 0, 0)\equiv L_2(\Rr_-, E_*)$$
and in the outgoing   one $$D_+=(0,0, L_2(\Rr_+, E))\equiv L_2(\Rr_+, E)$$
in the sense of the Lax-Phillips scattering theory  \cite{LP67},  as a first order differentiation operator on $v_-$ and $v_+$, respectively. 
 This operator, being the generator of the standard shift   semigroups on the half lines, { gives  a  justification   to } the name ``translational  form'' for {this   realisation }  of the dilation.
{    The ``middle'' operator $TU$ explicitly uses the  operator $A^*$,  to make  a  coupling  between the two terms $v_\pm(0)$  and of course between  both channels.}

The main result {of \cite{BMNW20}  is the Theorem 7.6:}
\begin{proposition}  \label{prop1}
The operator $\mathcal L_{tr}$ in the Hilbert space  $\mathcal H_{tr }=L_2(\Rr_-, E_*)\oplus H\oplus L_2(\Rr_+,E)$, defined in {Definitions \ref{def:L}  and \ref{def:op}} is a minimal selfadjoint dilation of the maximally dissipative operator $A$ in  $H$, i.e.
\begin{enumerate}
\item  $$ { \mathcal {L} _{tr}}= {  \mathcal  {L}_{tr}^*},$$

\item { for $U=(0,u,0)$
$$P_H(\mathcal {L}_{tr} - \lambda)^{-1}U=\left \{  \begin{array}{cc}  (A-\lambda)^{-1}u, & \lambda\in \C_-,\\ (A^*-\lambda)^{-1}u, & \lambda\in {  \mathbb C ^+,}
\end{array}  \right .
$$
where $P_H$ is the projection onto the second  component in $\mathcal H_{tr}$: $P_H(v_-,u,v_+)=
(0,u,0)$.}
\item
Define the completely non-selfadjoint subspace $H_{cns}$ of $A$ as in {\eqref{eq:cns}
and its orthogonal complement $H_{sa}$.}
Then the subspace $(0, H_{sa},0)\subset \mathcal H_{tr}$ is {a reducing subspace} for the dilation $\mathcal L_{tr}$,  and $\mathcal L_{tr}$ restricted  {to}  $(0,H_{sa},0)$ { is  }
$(0,A_{sa},0)   $, where $A_{sa}:=A|_{H_{sa}}$ is the selfadjoint part of $A$.  {Further  the  operator $\mathcal L_{tr}$ restricted to  {the second reducing subspace
$$ L_2(\Rr_-, E_*)\oplus H_{cns}\oplus L_2(\Rr_+, E),$$
the orthogonal complement of $(0, H_{sa},0)$,}
is the minimal selfadjoint dilation of $A|_{H_{cns}}$. }  Moreover
$$
\clos \left  ( \Span_{\lambda \not \in \Rr}  (\mathcal L_{tr}-\lambda)^{-1}
\left  ( \begin{array}{c}   L_2(\Rr_-,E_*)\\0\\L_2(\Rr_+, E)\end{array}  \right )  \right )= \left ( \begin{array}{c}   L_2(\Rr_-,E_*)\\H_{cns}\\L_2(\Rr_+, E)\end{array}
 \right ).
 $$
\end{enumerate}
\end{proposition}

Now we are ready to transform the {translational} form  of the dilation $\mathcal L_{tr}$, given by  Proposition   \ref{prop1}, into its unitarily equivalent spectral form.   The part of $\mathcal L_{tr}$, corresponding  to the completely non-selfadjoint component of $A$, is presented as the  multiplication operator by an independent variable $k\in \Rr$, in an $L_2$-space of vector-valued functions on $\Rr$.  { The existence of this  form {is clear from the Foias Theorem \ref{thm:Foias}}, but our translational form {describes it explicitly}. This  allows us to preserve  information about the original form of the operator $A$ under the requisite transformation. 

In what follows,   we will in some places assume that $A$ is completely non-selfadjoint. This will allow us to ignore the selfadjoint part  $A|_{H_{sa}}$ of $A$,  since in the minimal selfadjoint dilation this is  reflected by  the operator $A|_{H_{sa}}$  on $(0,H_{sa},0)$  
and can be   studied   using the classical spectral theorem for  selfadjoint operators in Hilbert space. Both parts, $A|_{H_{cns}}$  and $A|_{H_{sa}}$, should be considered  independently using these { different tools,  i.e.~the functional model  for the first part in $\mathcal H_{cns}$
and the standard spectral theorem  for the second  part in $\mathcal H_{sa}$.}
In view of this   we will mainly  concentrate  on the minimal selfadjoint dilation of a completely non-selfadjoint maximally dissipative operator $A$, without any  loss of generality.

{As the first step we consider two maps $\mathcal F_\pm$ transforming the translational form  of  the  Hilbert space $\mathcal H_{tr}=L_2 (\Rr_-, E_*) \oplus H\oplus L^2 (\Rr_+,E)$ into $L_2 (\Rr, E)$ and $L_2 (\Rr, E_*)$, respectively.}  

\begin{Definition}
For any vector $(v_-,u,v_+)\in \mathcal H_{tr}$ { and a.e.~$k\in\R$, set
$$ \left[\mathcal F_+\left ( \begin{array}{c} v_-\\u\\v_+\end{array}\right )\right] (k):=
-\dfrac1{\sqrt{2 \pi}} \Gamma (A-k+i0)^{-1}u + S^*(k) \hat v_-
(k)+\hat v_+(k)\in L_2(\Rr, E) $$
and
$$ \left[\mathcal F_-\left ( \begin{array}{c} v_-\\u\\v_+\end{array}\right )\right](k):=
-\dfrac1{\sqrt{2 \pi}} \Gamma_* (A^*-k-i0)^{-1}u + S(k) \hat v_+
(k)+\hat v_-(k)\in L_2(\Rr, E_*), $$
where $S(k):= S(k+i0)$ and $\hat v_\pm (k)$ are the  Fourier transforms of the vector-valued functions $v_\pm (\xi)$ extended  by $0$ onto the complementary  
semiaxis $\Rr_{ \mp}$:
$$\hat v_+ (k)=\dfrac1{\sqrt{2 \pi}}
\int_0^\infty e^{i \xi k} v_+(\xi)  d \xi, 
$$
 $$\hat v_-(k)=\dfrac1{\sqrt{2 \pi}}
\int_{-\infty}^0 e^{i \xi k} v_-(\xi)  d \xi. $$
}
\end{Definition}

{Using the canonical identification of an  analytic function from the  Hardy class in the upper or lower half-plane with its boundary values on the real line $\Rr$,
by the Paley-Wiener theorem \cite{Koosis,{SFBK10}}, we have}
$$  \hat v_{\pm}(k)\in H_2^ \pm := H_2 (\mathbb C^\pm).$$ 
{Similarly,   $\Gamma (A-k+i0)^{-1}  u$ and $\Gamma_*(A^*-k-i0)^{-1}u$  are the boundary values of $\Gamma (A-\lambda)^{-1}u$ and  $\Gamma_*( A^*- \lambda)^{ -1}u $,  where $\lambda \to k \mp i0$ in the lower and upper half-plane, respectively.}  The existence of the non-tangential boundary values {of the resolvents and $S(k+i0)$ and $S^*(k+i0)$,} in the strong {operator} topology of   the Hilbert  spaces $E$ and $E_*$ for a.e.~$k\in \mathbb R$   is guaranteed by the B.~Sz-Nagy Theorem \cite{SFBK10}.  Since the boundary values  { $S(k)$ and $S^*(k)$ are contractions,} we have that all three terms in   the formulae for $\mathcal F_\pm $ are { $L_2(\Rr,E)$- or  $L_2( \Rr,E_*)$-vector-valued} functions on $\Rr$. 
 Therefore the maps $\mathcal F_\pm $ are well  defined on the whole space $\mathcal H_{tr}$.  Moreover, {from Theorem \ref{thm:2.1},}  we have by the triangle inequality
 \begin{eqnarray*}
 &\norm{\mathcal F_+   \left (
 \begin{array}{c} v_-\\u \\ v_+  \end{array} \right ) }_{{L_2(\Rr,E)}}
  &\leq
 \norm{u}_H + \norm{S^* \hat v_-}_{L_2(\Rr,E)}+\norm{\hat v_+}_{L_2(\Rr,E)} \\
  &&\leq \norm{u}_H+\norm{\hat v_-}_{{L_2(\Rr,E_*)}}+\norm{ \hat v_+}_{L_2(\Rr,E)}
	=\norm{u}_H +\norm{v_-}_{{L_2(\Rr_-,E_*)}} +\norm{v_+}_{{L_2(\Rr_+,E)}}\\
  && \leq \sqrt3 \norm{ (v_-, u, v_+)}_{\mathcal H_{tr}}.
  \end{eqnarray*}
  Here we used   the Parseval identity \cite{Koosis} and    the contraction property of  $ S^*$.  The case of $ \mathcal F_-$  can be considered similarly. {Thus the maps are bounded  operators from $\mathcal H_{tr}$  to $L_2 (\Rr, E )$  or   to $L_2 (\Rr, E_*)$, respectively. }
 
  Following the ideas in the paper \cite{Nab81} we use the maps $\mathcal F_\pm$  for the construction  of  the spectral form of the dilation $\mathcal L_{tr}$. {We next  introduce a new Hilbert space:}
  $$\mathcal H_{sp}:=L_2 \left (\Rr, E\oplus E_*;  \left (\begin{array}{cc}  I_E  & S^*(k) \\ S(k) & I_{E_*} \end{array} \right )\right ).
  $$
 
  Our new version of the functional model Hilbert space $\mathcal H_{sp}$, referred to as Pavlov's symmetric form of the functional model,  is by definition {the closure in a weighted norm of the space of    vectors   of the form}
  $\left ( \begin{array}{c} \tilde g \\ g \end{array} \right )$, where $\tilde g = \tilde g (k)\in L_2 (\Rr, E)$  and $g=g(k)\in L_2 (\Rr, E_*)$.  The norm of the vector in $\mathcal H_{sp}$ is defined as follows.
  \begin{eqnarray*}
  &&  \norm{    \left ( \begin{array}{c} \tilde g \\ g \end{array} \right ) }_{\mathcal H_{sp}}^2:= \int_\Rr   \left \langle   \left  (
  \begin{array}{cc} I_E & S^*(k)\\ S(k) & I_{E_*} \end{array} \right ) \left (  \begin{array}{c} \tilde g \\ g \end{array} \right ) ,  \left ( \begin{array}{c}         \tilde g \\ g \end{array} \right )  \right \rangle
_{E\oplus E_*} dk.
\end {eqnarray*}
{Simple algebraic manipulations give the following identities.
  \begin{eqnarray}\label{eq:Hsp1}
    \norm{    \left ( \begin{array}{c} \tilde g \\ g \end{array} \right ) }_{\mathcal H_{sp}}^2&=&
 \norm{   \tilde g (k)+S^*(k) g(k)}^2_{L_2 (\Rr, E)}+
 \int_\Rr \langle (I-S(k) S^*(k) ) g(k), g(k)\rangle  _{E_*}\  dk \\
&=& \norm{ S(k)  \tilde g (k)+g(k)}^2_{L_2 (\Rr, E_*)}+
 \int_\Rr \langle (I-S^*(k) S(k) )   \tilde g (k),   \tilde g (k)\rangle  _{E}\  dk. \label{eq:Hsp2}
 \end {eqnarray}}
 Since the boundary values $S^*(k)\equiv ( S( k + i 0))^* ,  S(k)\equiv S(k+i 0)$ of the { contraction $S(z)$, $z \in \C_+$ are also contractions}, we obviously have
\be\label{eq:est}\norm  { \left ( \begin{array}{c} \tilde g \\ g \end{array} \right ) }_{\mathcal H_{sp}}\geq \norm{ \tilde g + S^*g}_{L_2 (\Rr, E)}\quad\hbox{  and }\quad \norm  { \left ( \begin{array}{c} \tilde g \\ g \end{array} \right ) }_{\mathcal H_{sp}}\geq \norm{S \tilde g + g}_{L_2 (\Rr, E_*)},
\ee   
where we omitted the   arguments of the functions on the right hand side terms for  notational convenience.  The elements of $\mathcal H_{sp}$     are the limits    of vectors from $L_2(\Rr, E))\oplus L_2(\Rr, E_*)$   and  we  will still denote them as $(\tilde g(k), g(k))$,  although {this is symbolic, especially since the matrix  weight
$\left  (
  \begin{array}{cc} I_E & S^*(k)\\ S(k) & I_{E_*} \end{array} \right )  $
  may be degenerated at a set of positive Lebesgue measure  on $\Rr$}.  On the other hand, { due to \eqref{eq:est},} the expressions $\tilde h(k):=  \tilde g (k) + S^* (k) g(k)$  and   $\  h(k):=  S(k) \tilde g (k) +  g(k)$  are still    $L_2(\Rr, E)$  and $L_2 (\Rr, E_*)$  functions respectively, even after taking a closure.
  Note that the   $L_2$-vector valued functions $(\tilde h(k), h(k))$   form a  de Branges \mm{\cite{dB68}} version of the functional model \cite{NV86}.
		
    Alternatively, {the pairs} 
   $$
   (  \tilde g(k)+S^*(k) g(k), g(k)  )
  = (  \tilde h (k), g(k))\in L_2(\Rr, E)\oplus L_2 (\Rr, E_*;(I_{E_*}-S(k)S^*(k))^{1/2}) 
  $$
  and
  $$
 ( S(k) \tilde g(k)+  g(k), { \tilde g(k)}  )
  =  (   h (k), \tilde g(k))\in  L_2(\Rr,E_*) \oplus  L_2(\Rr, E;   (I_E-S^*(k)S(k))^{1/2})
  $$
  give a transformation to the Sz-Nagy-Foias form of the functional model \cite{SFBK10}.
  We will discuss this in  more detail later.
 
{
\begin{lemma}\label{lem:uniquedet}
  A vector
  $\left ( \begin{array}{c} \tilde g \\ g \end{array} \right )\in \mathcal H_{sp}$
  is uniquely determined by the two vector functions $\tilde h(k)=\tilde g(k)+S^*(k)g(k)$  and $h(k)=S(k) \tilde g(k)+ g(k)$.
\end{lemma}
}
\begin{proof}
  Assume that {$\tilde h(k)=0$  and  $h(k)=0$}  for almost every $k\in \Rr$   and  recall that the operator functions $S(k)$ and $S^*(k)$  are also only well defined for a.e. $k\in \Rr$.
  Then {almost everywhere}
  $$0=  h(k)-S(k)\tilde h(k)={S(k)\tilde g(k) + g(k)-S(k)(\tilde g(k) +S^*(k) g(k))}=
  (I_{E_*}-S(k)S^*(k))g(k),$$
  and  therefore
$  (I_{E_*}-S(k)S^*(k))^{1/2}g(k)=0$
almost everywhere.
Now comparing  this    with {\eqref{eq:Hsp1}},
we have that
$\norm{ \left ( \begin{array}{c} \tilde g \\ g \end{array} \right ) }_{ \mathcal H_{sp}}= 0$.  { This   is merely a formal proof, but it  can be made   rigorous  via      a limiting argument as we now show. }

{Let }
 $$      \lim_{n\to\infty }  \left ( \begin{array}{c} \tilde g_n(k)   \\   g_n(k)
 \end{array} \right )  =  {\left ( \begin{array}{c} \tilde g(k) \\  g(k) \end{array}
 \right )}
 $$
{ in $\cH_{sp}$ with
 $\tilde g_n(k)\in L_2 (\Rr, E)$ and $ g_n(k) \in L_2 (\Rr, E_*)$. Then setting
$ \tilde h_n:=\tilde g_n+ S^*g_n$ and $h_n :=S\tilde g_n + g_n,$  we have
 $\tilde h = \lim_ {n  \to\infty } {\tilde h_n  } =0$  
   and  $h=\lim_{n\to \infty}  h_n=0$   in the norm of $L_2(\Rr,E)$ and $L_2(\Rr, E_*)$,  respectively. So,}
 \begin{eqnarray*}
  0&=& h -S\tilde h=
 \lim_{n\to \infty}( h_n - S\tilde h_n)\\
&=& \lim_{n \to  \infty } ( S\tilde g_n + g_n - S( \tilde g _n+ S^*g_n))\\
&=&    \lim_{n \to  \infty } (  g_n - SS^*  g_n)=
{ ( I_{E_*}-SS^*)^{1/2}  (  I_{E_*}-SS^*)^{1/2}g.}
  \end{eqnarray*}
  { So $(I_{E_*}-SS^*)^{1/2} g=0$  as a function   from  $L_2 (\Rr, E_*)$,} {and the lemma follows from \eqref{eq:Hsp1}, as above.}
	\end{proof}
   
 {   In what follows we will usually omit  this type of   argument based on a limiting  procedure of approximating {elements 
  $\left(  \begin{array}{c} \tilde g \\g\end{array}\right ) \in\cH_{sp}$ by $L_2$-vector-valued functions
    and instead proceed   formally as we indicated in the proof of the lemma}.   However,   we note that a rigorous proof  along the indicated lines 
 can  always be  performed.}
  {Operating} with the symbol $\left ( \begin{array}{c} \tilde g \\g\end{array}\right )$  (in Pavlov's symmetric form) is often more convenient,  especially  taking into consideration that the $L_2$-vector-valued  functions $\tilde h(k)$ and $h(k)$ are also {only defined for a.e.~$k\in\R$.}

 Now we are ready to formulate the main  result.

 \begin{theorem}  \label{thm4.5}
 { Let $\mathcal L_{tr}$ be   a  minimal selfadjoint dilation of a completely non-selfadjoint maximally dissipative operator $A$  in a Hilbert space $H$.}
  Then there exists a unique  transformation {$\Phi$} of  the translational form Hilbert space $\mathcal H _{tr}=L_2 (\Rr, E_*) \oplus H \oplus L_2(\Rr_+,E)$ onto  the spectral form  of the Hilbert space
  $$\mathcal H_{sp}=L_2 \left (\Rr, E\oplus E_*;  \left (\begin{array}{cc}  I_E  & S^*(k) \\ S(k) & I_{E_*} \end{array} \right )\right )
  $$
  {such that the image in $\mathcal H_{sp}$ of the translational    form
  of the selfadjoint dilation $\mathcal L_{tr}$  is the multiplication operator by the independent variable  $k \in  \Rr$
	satisfying the explicit formula
 \begin{equation}
 \left \{
    \begin{array}{c}
  \tilde g(k)+ S^*(k)g(k) =\mathcal F_+(v_-, u, v_+), \\
 S(k) \tilde g(k)+g(k) =\mathcal F_-(v_-, u, v_+),
  \end{array} \right .  
  \label{eq10}
  \end{equation}}
	{where $\Phi(v_-,u,v_+)=\left ( \begin{array}{c} \tilde g \\g\end{array}\right )$.}
  \end{theorem}
\begin{remark} A slightly more explicit formula for $\Phi$ is given in Lemma  \ref{lem4.2} below. \end{remark}
  \begin{proof}  {The long proof of the theorem is broken down into five steps.}
	
 \underline{\textbf{Step 1:}} Density of a set of test vectors.
   
  {We formulate the result as a  lemma.}
  \begin{Lemma} \label{lem4.1}
 Let $A$  be a maximally dissipative operator in $H$. Then
   the  linear set of test functions in the translational  version of the Hilbert space $\mathcal H_{tr}$
 $${\tau:=}\ \Span  \left \{
 \Span_{v_+\in L_2(\Rr_+,E),v_-\in L_2(\Rr_- ,E_*),e\in E,\lambda \in \mathbb{C}_+} \left  ( \begin{array}{c} v_-\\(\Gamma (A+\lambda)^{-1})^*e\\v_+
 \end{array}
 \right )  ,\right .
 $$
  $$ \left. \Span_{v_+\in L_2(\Rr_+,E),v_-\in L_2(\Rr_-,E_*),e_*\in E_*,\mu \in \mathbb {C}_-}
  \left   (\begin{array}{c} v_-\\(\Gamma_* (A^*+\mu)^{-1})^*e_*\\v_+\end{array} \right)
  \right \}
 $$
  is a dense set in
  $\mathcal H\ominus (0,H_{sa},0)  $  where $H_{sa} $  is the selfadjoint part of    the Langer   decomposition of $H$.
 In particular, if $A$ is a completely non-selfadjoint operator
 then the set $\tau$ presented above  is dense in the whole space $\mathcal H_{tr}$.
 \end{Lemma}
  \begin{proof}
 Consider a vector $\left ( \begin{array}{c}  f_-\\u\\f_+\end{array} \right ) \in \mathcal H_{tr} $ orthogonal to {$\tau$}.  Then, choosing $e=0$  or $e_*=0$ we get immediately that
  $f_-=0$ and $f_+=0$.  On the other hand, putting $v_-=0$ and $v_+=0$ we obtain from the orthogonality condition that
  \begin{enumerate}
  \item
  $\Gamma(A+\lambda)^{-1}u=0$ for all  $\lambda\in  \mathbb C_+$,
  \item
  $\Gamma_*(A^*+\mu)^{-1}u=0$ for all $\mu\in  \mathbb C_-$.
 \end{enumerate}
 By Theorem {\ref{thm:langer}}, conditions
 (1) and (2) together imply  that $u\in H_{sa}$.
 \end{proof}
 
 \underline{\textbf{Step 2:}} Embedding of the test vectors into $\mathcal H_{sp}$.

 {Let us first define the transformation $\Phi$ acting on the test vectors from Lemma \ref{lem4.1}  by the explicit formula  in the following lemma:}
\begin{Lemma}\label{lem4.2}
{Let $\lambda_0\in \mathbb{C}_+$, $e \in E$, $v_-\in L_2(\Rr_-, E_*)$ and $v_+ \in L_2 (\Rr_+,E)$.
The map  
$$
\Phi:\left (\begin{array}{c} v_-\\ (\Gamma (A-\overline \lambda_0)^{-1} )^*e\\v_+\end{array}\right )\to {\left  ( \begin{array}{l} \tilde g(k)\\ g(k)
 \end{array} \right )}:=\left ( \begin{array}{c} \hat v_+(k)+\dfrac i{\sqrt{ 2 \pi}} \dfrac e{k-\lambda_0}\\    \hat v_-(k)-\dfrac i{\sqrt{ 2 \pi}} \dfrac{  S(\lambda_0)e}{k-\lambda_0}\end{array}
\right )
$$
satisfies condition (\ref{eq10}).   }
\end{Lemma}
\begin{proof}
We have
\begin{eqnarray*}
S(k)\tilde g(k) + g(k) &=& S(k)\left( \hat v_+(k) + \dfrac i {\sqrt{2 \pi}} \dfrac e{k-\lambda_0}\right) + \hat v_- (k)-\dfrac i {\sqrt{ 2 \pi} }\dfrac { S(\lambda_0)e}{k-\lambda_0} \\
& =&S(k) \hat v_+(k) + \dfrac i { \sqrt {2 \pi}}\dfrac { (S(k)-S(\lambda_0))e}{  k-\lambda_0} + \hat v_-(k),
\end{eqnarray*}
and
\begin{eqnarray*}
 \mathcal F_-  \left(
\begin{array}{c}
v_-\\ (\Gamma (A-\overline \lambda _0)^{-1})^* e\\v_+
\end{array}
\right )  
&=&-\dfrac1{ {\sqrt 2 \pi}}\left ( \Gamma_*  { (A^*-k- i 0)}^{-1} \right )\left (\Gamma (A-\overline\lambda_0)^{-1}\right )^* e+S(k)\hat v_+(k)+\hat v_-(k).
\end{eqnarray*}
{To prove that the two are equal,} we need to show
\be\label{eq:11} \dfrac{  S(k)-S(\lambda_0)}{k-\lambda_0} e=i(\Gamma_*{(A^*-k-i0)}^{-1}) (\Gamma (A-\overline \lambda_0)^{-1})^*e\ee
for all $e \in E$ and a.e.$ k \in  {\mathbb R} $.
In order to do this, we use {\eqref{eq:Sdiff}
and set} $\tilde
\mu=\lambda_0 \in \mathbb{C}_+, \mu = k+i\eps, \eps >0.$
{Letting $\eps\to 0$  we have that for fixed $e\in E$ and a.e. $k\in \mathbb R$ the equality \eqref{eq:11}  is valid.}  We remind the reader that, { applied to any $ e \in E$  the right hand side term in \eqref{eq:Sdiff} lies in the vector-valued Hardy class as a   function of  $\mu$ by Theorem \ref{thm:2.1}}.  
Similarly,
$$\mathcal F_+  \left (   \begin{array}{c}  v_-\\  (\Gamma (A-\overline \lambda_0)^{-1})^*e \\v_+ \end{array}  \right )=
-\dfrac1{ { \sqrt 2\pi}}  ( \Gamma (A-k+i0)^{-1})  (\Gamma (A-\overline \lambda_0)^{-1})^*e+ S^*(k)
\hat v_- (k) + \hat v_+(k).
$$
In order to prove that {this is equal to $\tilde g(k)+ S^* g(k)$,} we have    to show
$$
\dfrac{  I_E-S^*(z) S(\lambda_0)}{\overline z-\lambda_0}=i (\Gamma (A-\overline z)^{-1})  (\Gamma{ (A-\overline \lambda _0)^{-1})^*}
$$
for all $\lambda_0, z \in \mathbb{C}_+$, which is exactly \eqref{eq:SStarS1}.
As   $\eps  \to 0$    we see that $z= k+i \eps \to k$ .  Thus  both of the terms above converge     in the strong topology  for  a.e.~$k \in \Rr$.   This  completes  the proof.
\end{proof}
\begin{remark}\label{unlabeledp41}
{From \cite[Lemma 5.4]{BMNW20}, it is easily seen that}
the test vectors of first type from  Lemma \ref{lem4.1}  belong  to $\mathcal {D(L}_{tr})$  provided that $v_+\in H^1 (\Rr_+,E)$, $v_-\in H^1 (\Rr_-, E_*)$ with $v_+(0)=-e$ and $v_-(0)={-S( \lambda_0)}e$,  {in particular,} $ v_-(0) = {S( \lambda_0)}v_+(0)$.

Concerning  the second type of test vectors  $(v_-, ( \Gamma_*(A^*-\overline \mu_0)^{-1})^*e_*, v_+)\in \mathcal H_{tr}$
with $e_*\in E_*, \mu_0\in \mathbb{C}_-$ { from Lemma \ref{lem4.1}, we have that they lie in} $\mathcal {D(L}_{tr})$  provided $v_-(0)= -e_*$ and $v_+(0)={S^*(\overline \mu_0)}v_-(0)$.  
\end{remark}
\begin{Lemma}  \label{lem4.3}
Define the map
$$
\Phi:  \left ( \begin{array}{c} v_-\\ {(\Gamma_*  (A^*-\overline  \mu_0)^{-1})^*} e_*\\ v_+ \end{array} \right )  \to {\left  ( \begin{array}{l} \tilde g(k)\\ g(k)
 \end{array} \right )}
 =\left  ( \begin{array}{l}
\hat v_+(k)+\dfrac {i}  {\sqrt{2 \pi}}  \dfrac { S^* (\overline \mu_0)e_*}{k-\mu_0}  \\ { \hat v_-(k)-\dfrac {i}  {\sqrt{2 \pi}}  \dfrac { e_*  }{k-\mu_0}}
\end{array} \right ) $$
for any $\mu_0\in \mathbb{C}_-, e_*\in E_*, v_+\in L_2 (\Rr_+, E), v_-\in L_2 (\Rr_-E_*)$.  {This map satisfies the condition {\eqref{eq10}}.}
\end{Lemma}

\begin{proof}
The statement of the lemma  and its proof are both completely analogous to those of Lemma \ref{lem4.2}. {Therefore we omit the details and simply note that
 in the proof  we use the following identities}
$$\dfrac{  S(\overline \mu_0) -S(\overline z)}{  \overline z-\overline \mu_0}=-i(\Gamma_*(A^*-\overline \mu_0)^{-1}) (\Gamma (A-z)^{-1})^*
$$  with $z, \mu_0\in \mathbb C_-$ {which is obtained from \eqref{eq:Sdiff},}  and
$$\dfrac{S(z)  S^*(\overline \mu_0) -I_{E_*}}{    z- \mu_0}=i(\Gamma_*(A^*-z)^{-1}) (\Gamma_* (A^*-\overline \mu_0)^{-1})^*
$$
with   $z \in \mathbb{C}_+$, $\mu_0\in \mathbb{C}_-$, {which is obtained from \eqref{eq:SStarS2} and noting $S(z)=S^*_*(\overline{z})$ from Lemma \ref{lemma:adj}}. 
\end{proof}

{In order to avoid unnecessary notation we have used  the symbol   $\Phi$ to denote both the maps from Lemmas \ref{lem4.2} and \ref{lem4.3}. } 
{The construction of the map   $\Phi$ allows us to extend it by linearity to any finite linear combination of test functions  of both  of the above  types. We will see  in the next lemma that this procedure does not lead to any  contradiction. Moreover, it  justifies  the  use of the same symbol $\Phi$ for both maps.}   In view of this, the extended embedding maps of the test vectors into the spectral  version's  Hilbert space $\mathcal H_{sp}$  have the form
\begin{equation}
\Phi \left ( \begin{array}{c} v_-\\\sum_j (\Gamma (A-\overline \lambda_j)^{-1})^*e_j\\v_+\end{array}\right )
\to
 \left ( \begin{array}{c}  \hat v_+(k) + \dfrac i {\sqrt{ 2 \pi }} \sum_j \dfrac {e_j}{k-\lambda_j} \\ \hat v_-(k) - \dfrac i {\sqrt{ 2 \pi }} \sum_j \dfrac {S(\lambda_j) e_j}{k-\lambda_j}\end{array}\right )  \label{em1} ,
\end{equation}
$v_-\in L_2 (\Rr_-, E_*)$, $v_+\in L_2 (\Rr_+,E)$, $\lambda_j\in \mathbb{C}_+$, $e_j\in E$, $j=1,2...N$
and
\begin{equation}
\Phi \left ( \begin{array}{c} v_-\\\sum_j (\Gamma_* (A^*-\overline \mu_j)^{-1})^*e_{*j}\\v_+\end{array}\right )
\to
 \left ( \begin{array}{c}  \hat v_+(k) + \dfrac i {\sqrt{ 2 \pi }} \sum_j \dfrac {S^*(\overline \mu_j)e_{*j}}{k-\mu_j} \\ \hat v_-(k) - \dfrac i {\sqrt{ 2 \pi }} \sum_j \dfrac {e_{*j}}{k-\mu_j}
 \end{array}\right )  \label{em2}
\end{equation}
$v_-\in L_2(\Rr_-, E_*)$, $v_+\in L_2(\Rr_+,E)$, $\mu_j\in \mathbb{C}_-$, $e_{*j}\in E_*$, $j=1,2,...N$.

\underline{\textbf{Step 3:}} {Isometry property of the embedding maps.}

\begin{Lemma}  \label{lem4.4}
The embedding maps {given by \eqref{em1} and \eqref{em2} are isometries. By linearity, they   generate  the isometry map $\Phi$ from a dense subset of $\mathcal H_{tr}\ominus (0,H_{sa},0)$ 
to $\mathcal H _{sp}$.}
\end{Lemma}

\begin{proof}
{From Lemma \ref{lem4.1}, we know that the linear set generated by test vectors of both types is dense in $\mathcal H_{tr}\ominus (0,H_{sa},0)$.
Therefore it remains to show the isometry property and that \eqref{em1} and \eqref{em2} do not lead to a contradiction.}

Consider two test vectors {$\left ( \begin{array}{c} v_-\\ \sum_{j=1}^N (\Gamma (A-\overline \lambda_j)^{-1})^*e_j\\ v_+\end{array}\right )$ and $\left ( \begin{array}{c} w_-\\ \sum_{m=1}^M (\Gamma (A-\overline \mu_m)^{-1})^*f_j\\ w_+\end{array}\right )$
as on the left hand side of (\ref{em1}).
In particular, we assume here that  $\lambda_j \in \mathbb{C}_+$, $j=1,...,N$   and $\mu_m \in \mathbb{C}_+$, $m=1,2,...,M$  and that the vectors $e_j, f_m \in E$ for all values of $j$ and $m$. Let 
$\left ( \begin{array}{l} \tilde  g \\ g
\end{array} \right )
$
and   $\left ( \begin{array}{l} {\tilde{f}} \\ {f} \end{array} \right ) $   denote  
their images under $\Phi$, respectively. }

Then
\begin{eqnarray}
&&\left  \langle  \left ( \begin{array}{l} \tilde  g \\ g \end{array} \right ), \left ( \begin{array}{l} {\tilde{f}} \\ {f} \end{array} \right )\right  \rangle _{\mathcal H_{sp}}=
\left \langle  \left (  \begin{array}{cc}I_E & S^*\\ S& I_{E_*} \end{array} \right ) \left ( \begin{array}{l} \tilde  g \\ g \end{array} \right ),\left  ( \begin{array}{l} {\tilde{f}} \\ {f} \end{array} \right ) \right \rangle_{L_2(\Rr,E\oplus E_*)} \nonumber \\
&&   =\llangle \hat v_+(k) + \dfrac i { \sqrt{ 2 \pi}} \sum_j \dfrac {e_j}{k-\lambda_j} + S^*(k)\hat v_-(k)- \dfrac i { \sqrt{ 2 \pi}} \sum_j \dfrac { S^*(k) S(\lambda_j) e_j}{k-\lambda_j}, { \hat w_+(k) + \dfrac i { \sqrt{2 \pi}} \sum_m \dfrac{f_m}{k-\mu_m} }\rrangle _{L_2 (\Rr,E)}\nonumber  \\
&&
+  \llangle S(k) \hat v_+(k)+\dfrac i { \sqrt{2 \pi}} \sum_j \dfrac{S(k)e_j}{k-\lambda_j}+ \hat v_-(k) - \dfrac i { \sqrt{2 \pi}} \sum_j \dfrac{S(\lambda_j)e_j}{k-\lambda _j}, {\hat w_-(k) - \dfrac i { \sqrt{2 \pi}} \sum_m\dfrac{S(\mu_m)f_m}{k-\mu_m}}\rrangle_{L_2 (\Rr, E_*)}.   \label{e15}
\end{eqnarray}
{This follows   since images of  ${\Phi}$ consist of  images of finite   linear combinations of test vectors of the first type   belonging  to  $L_2 (\Rr,E)\oplus L_2 (\Rr, E_*)\subset \mathcal H_{sp}$.}

To continue the explicit calculation of the last expression  (\ref{e15})
  notice that by the Paley-Wiener Theorem \cite{Koosis},
$$  \hat v_+{(k)}=\dfrac 1 { \sqrt{ 2 \pi}} \int_{\Rr_+}  e^{ik\xi} v_+(\xi) d \xi \in H_2^+ (E),$$
$$\hat v_-{(k)}=\dfrac 1 { \sqrt{ 2 \pi}} \int_{\Rr_-}  e^{ik\xi} v_-(\xi) d \xi \in H_2^- (E_*),$$
and
\begin{equation}
\dfrac{ S(k) e_j}{k-\lambda_j}=\dfrac{  (S(k)-S(\lambda_j))e_j}{k-\lambda_j}  +
\dfrac{ S(\lambda_j)e_j}{k-\lambda_j}  \label{eq:pmdecomp}
\end{equation}
gives the orthogonal decomposition of the vector  from $L_2 (\Rr,E_*)$ into the sum of two vector functions from $H^+_2 (E_*)$ and $H^-_2(E_*)$,    respectively.

Then, {omitting the index on the scalar  product  
in the Hilbert spaces $L_2 (\Rr, E)$  and $L_2(\Rr,E_*)$ 
 and using  its  linearity properties, 
 we get that} 
{\begin{eqnarray}\nonumber
\left  \langle  \left ( \begin{array}{l} \tilde  g \\ g \end{array} \right ), \left ( \begin{array}{l} {\tilde{f}} \\ {f} \end{array} \right )\right  \rangle _{\mathcal H_{sp}}&=&   \langle \hat v_+, \hat w_+ \rangle  + \langle \hat v_-, \hat w_- \rangle  + \dfrac 1{2 \pi} \left  \langle \sum_j \dfrac{e_j}{ k-\lambda_j},    \sum_m \dfrac{f_m}{ k-\mu_m}\right  \rangle-\dfrac 1{2 \pi} \left  \langle \sum_j \dfrac{S^*S(\lambda_j) e_j}{ k-\lambda_j},   \sum_m \dfrac{f_m}{ k-\mu_m} \right \rangle\\
&&+ \left \langle S^*(k) \hat v_-,\dfrac{ i}{  \sqrt{ 2\pi}} \sum_m \dfrac{ f_m}{k-\mu_m}\right  \rangle + \dfrac{ i}{  \sqrt{ 2\pi}} \left \langle  \hat v_-,\dfrac{ S(\mu_m) f_m}{k-\mu_m}\right \rangle.\label{eq:gf}
\end{eqnarray}}
Here we have used that
\begin{enumerate}
\item
$\hat v_+\in H^+_2 (E) \perp \sum_m \dfrac{f_m}{k-\mu_m}\in H^-_2(E)$,
\item
$\sum_j     \dfrac{e_j}{k-\lambda_j}\in H^-_2(E)\perp
\hat w_+ \in  H^+_2 (E) $,
\item
$S^*(k)\hat v_-\in H^-_2(E) \perp  \hat w_+\in  H^+_2(E) $,
\item
$\sum_j     \dfrac{S^*S(\lambda_j)e_j}{k-\lambda _j}\in H^-_2(E)\perp
\hat w_+ \in  H^+_2 (E), $
\item
$  S(k) \hat v_+\in H^+_2(E_*) \perp  \left(  \hat w_-  -\dfrac{ i}{  \sqrt{ 2\pi}} \sum_m \dfrac{ S(\mu_m ) f_m}{k-\mu_m}\right) \in H^-_2(E_*),
$
\item
$  \sum_j  \dfrac{S(k) e_j}{k-\lambda_j}-  \sum_j  \dfrac{S(\lambda_j) e_j}{k-\lambda_j}\in H^+_2(E_*)\perp  \left(  \hat w_-  -\dfrac{ i}{  \sqrt{ 2\pi}} \sum_m \dfrac{ S(\mu_m ) f_m}{k-\mu_m} \right) \in H^-_2(E_*)$; this is
 due to the decomposition (\ref{eq:pmdecomp}).
\end{enumerate}
The last two terms in \eqref{eq:gf} cancel because
\begin{eqnarray*}
  \left \langle  S^*(k)\hat v_-, \dfrac{ i}{  \sqrt{ 2\pi}} \sum_m \dfrac{  f_m}{k-\mu_m}  
\right \rangle  &=&   -\dfrac{ i}{  \sqrt{ 2\pi}}
\left \langle    \hat v_-,  \sum_m \dfrac{ S(k) f_m}{k-\mu_m} \right \rangle\\  
&=&
-\dfrac{ i}{  \sqrt{ 2\pi}}
\left \langle    \hat v_-,  \sum_m \dfrac{ S (\mu _m)f_m}{k-\mu_m} \right \rangle 
-\dfrac{ i}{  \sqrt{ 2\pi}}
\left \langle    \hat v_-,  \sum_m \dfrac{ ( S(k) -S (\mu_m))f_m}{k-\mu_m} \right \rangle\\ 
&=& -\dfrac{ i}{  \sqrt{ 2\pi}}
\left \langle    \hat v_-,  \sum_m \dfrac{ S(\mu_m)  f_m}{k-\mu_m} \right \rangle
\end{eqnarray*}
by the  orthogonality of $H^-_2(E_*)$ and $H^+_2(E_*)$.

{We have
\begin{eqnarray*}
\left \langle  
\dfrac{S^*(k)S(\lambda_j))e_j}{k-\lambda_j},  \dfrac{f_m}{k-\mu_m}
\right \rangle  
& =&\left \langle  
\dfrac{S(\lambda_j) e_j}{k-\lambda_j}, \dfrac{S(k) f_m}{\lambda-\mu_m} \right  \rangle \\
& = &\left \langle \dfrac{ S(\lambda_j) e_j}{k-\lambda_j}, \dfrac{ ( S(k) - S(\mu_m)) f_m}{k-\mu_m}  \right \rangle  
+\left \langle  
\dfrac{S(\lambda_j) e_j}{k-\lambda_j}, \dfrac{S(\mu_m) f_m}{k-\mu_m} \right  \rangle \\
&=&
\left \langle\dfrac{S(\lambda_j) e_j}{k-\lambda_j}, \dfrac{ S(\mu_m) f_m}{k-\mu_m} \right  \rangle\ =\
\left \langle\dfrac{S^*(\mu_m)S(\lambda_j)e_j}{k-\lambda_j}, \dfrac{ f_m}{k-\mu_m} \right  \rangle
\end{eqnarray*}
by \eqref{eq:pmdecomp} and the orthogonality of $H^+_2(E_*)$  and   $H^-_2(E_*).$}
 
{Therefore,  applying the Parseval identity, \eqref{eq:gf} becomes
\begin{eqnarray*}
\left  \langle  \left ( \begin{array}{l} \tilde  g \\ g \end{array} \right ), \left ( \begin{array}{l} {\tilde{f}} \\ {f} \end{array} \right )\right  \rangle _{\mathcal H_{sp}}&=&   \langle v_+, w_+\rangle + \langle v_-, w_-\rangle
+ \dfrac1{2 \pi}  \left \langle   \sum_j
\dfrac{(I_E-S^*(k)S(\lambda_j))e_j}{k-\lambda_j} , \sum_m  \dfrac{f_m}{k-\mu_m}\right \rangle\\
&=&
\langle v_+, w_+ \rangle + \langle v_-, w_- \rangle  + \dfrac1{2 \pi}  \sum_{j ,m} \left \langle \dfrac{(I_E-S^*(\mu_m)S(\lambda_j))e_j}{k-\lambda_j} ,  \dfrac{f_m}{k-\mu_m}
\right \rangle.
\end{eqnarray*}}

An explicit calculation of the residues yields
$$\dfrac1{2\pi i} \left \langle  \dfrac1{k-\lambda_j} ,\dfrac1 {k-\mu_m} \right  \rangle_{L_2(\R)}
=\dfrac1{2 \pi i} \int_\Rr  \dfrac {dk}{  (k-\lambda_j)(k-\overline\mu_m)}= \frac 1 {\lambda_j- \overline\mu_m}   \;{\rm  for \;}  \lambda_j,\mu_m
\in \mathbb{C}_+.$$
Hence,
\begin{eqnarray*}
\left \langle \left ( \begin{array}{l} \tilde g\\ g\end{array}\right ), \left (  \begin{array}{l}{\tilde{f}}\\ {f}\end{array}\right )\right  \rangle_{\mathcal H_{sp}}
&=&\langle v_+, w_+\rangle+\langle v_-, w_-\rangle+
\sum_j\sum_m \dfrac i {\lambda_j -\overline \mu_m}
\langle (I_E-S^* (\mu_m) S(\lambda_j)e_j, f_m\rangle_E\\
& =&  \left \langle  \left (
 \begin{array}{c} v_-\\ \sum_j(  \Gamma (A-\overline \lambda _j)^{-1})^* e_j\\ v_+\end{array} \right ),
\left (    \begin{array}{c} w_-\\ \sum_m (\Gamma (A-\overline \mu_m)^{-1})^* f_m\\ w_+\end{array} \right )
\right  \rangle_{\mathcal H_{tr}}
\end{eqnarray*}
due to { \eqref{eq:SStarS1}} with $w={\mu_m}$ and $z= \lambda_j$.

Similarly one proves that the
second map, {given in \eqref{em2},}
  is also isometric. The proof does not differ essentially from the previous one and will  be omitted.  We note, more generally   that the proof also    follows from {interchanging} $A$ and $A^*$.
 
  To {complete the proof of} Lemma  \ref{lem4.4}  it is sufficient to consider two test vectors which are an arbitrary linear combination  of  vectors of  the first and  second type:  
  $$ {\overrightarrow{G}} :=  \left (    \begin{array}{c} v_-\\ \sum_j (\Gamma (A-\overline \lambda _j)^{-1})^* e_j+ \sum_m (\Gamma_* (A^*-{\overline\eta} _m)^{-1})^* e_{*m} \\ v_+\end{array} \right )
  $$
  with $v_- \in L_2 (\Rr_-, E_*)$, $v_+\in L_2 (\Rr_+,E)$, $\lambda _j\in \mathbb{C}_+$, $e_j \in E$  and $ {\eta}_m\in \mathbb{C}_-$, $e_{*m} \in E_*$ for all $j$ and $m.$
 
 Decomposing the vector  ${\overrightarrow{G}}  $
 into the sum of two vectors of the first and second type separately  we get
 $$  {\overrightarrow{G}}  :=
  \overrightarrow{{G}_1}+  \overrightarrow{{G}_2}  :=
  \left (    \begin{array}{c}
  0\\
  \sum_j (\Gamma(A-\overline \lambda_j)^{-1})^* e_j\\
  0
  \end{array} \right )
  +
   \left (    \begin{array}{c}
  v_-\\
  \sum_m (\Gamma_*(A^*-{\overline\eta} _m)^{-1})^* e_{*_m}\\
  v_+
  \end{array} \right ).
  $$
  Next we define
  $$ \Phi  {\overrightarrow{G}} :=
  \Phi  \overrightarrow{{G}_1}+\Phi\overrightarrow{{G}_2}=
 \left (    \begin{array}{l}
  \dfrac i { \sqrt{ 2 \pi}}  \sum_j  \dfrac{e_j}{k-\lambda_j} \\
 -\dfrac i { \sqrt{ 2 \pi}}  \sum_j \dfrac {S( \lambda_j) e_j}{k-\lambda_j}
   \end{array} \right )
    +
    \left (    \begin{array}{l}
 \hat v_+(k)  +  \dfrac i { \sqrt{ 2 \pi}}  \sum_m \dfrac{S^*({\overline\eta}_m)e_{*m}}{k-{\eta}_m}\\
\hat v_-(k) -\dfrac i { \sqrt{ 2 {\pi}} } \sum_m \dfrac { {e_{*  m}}}{ k-{\eta}_m}
    \end{array} \right ),
    $$
    while a similar notation  will be used for
    the second vector {$\overrightarrow{F}  :=
  \overrightarrow{F_1} + \overrightarrow{F_2}
  $}.
 {To show that $\Phi$ is well-defined, it is sufficient to show that it has  the isometry property.}
 
  By linearity it is enough to consider just one term of different types in each sum over, i.e.
  {$$
  \overrightarrow{{G}_1 } :=\left (    \begin{array}{c}
  0 \\
(\Gamma (A-\overline \lambda)^{-1})^* e \\
 0
   \end{array} \right )
  \quad {\rm and } \quad
      \overrightarrow{{G}_2 } :=\left (    \begin{array}{c}
 v_-  \\
 (\Gamma_*  (A^*- \overline \eta)^{-1})^* e_{*}\\
 v_+
   \end{array} \right ),
   $$}
   and similarly
    {$$
  \overrightarrow{F_1} :=\left (    \begin{array}{c}
 0 \\
 (\Gamma (A-\overline \mu)^{-1})^* f\\
 0
   \end{array} \right )
  \quad {\rm and } \quad
      \overrightarrow{F_2} :=\left (    \begin{array}{c}
 w_-  \\
 (\Gamma_*  (A^*- \overline \nu)^{-1})^* f_{*} \\
 w_+
   \end{array} \right ),
   $$}
	{with $\mu,\nu\in\C_-$. Now, using that $\Phi$ is an isometry on each type of test vectors individually,}
  \begin{eqnarray*}
   \langle \Phi (  \overrightarrow{{G}_1 }+\overrightarrow{{G}_2 }),
   \Phi (  \overrightarrow{{F_1}}+\overrightarrow{{F_2} } )
   \rangle_{\mathcal H_{sp}}
    &=&\langle \Phi   \overrightarrow{{G}_1 }, \Phi\overrightarrow{{F_1}}\rangle_{\mathcal H_{sp}}
   +
   \langle \Phi   \overrightarrow{{G}_2 },  \Phi\overrightarrow{{F_2}} \rangle_{\mathcal H_{sp}}
    + \langle \Phi   \overrightarrow{{G}_1 },  \Phi\overrightarrow{{F_2}}\rangle_{\mathcal H_{sp}}
   +
   \langle \Phi   \overrightarrow{{G}_2 }, \Phi\overrightarrow{{F_1}}\rangle_{\mathcal H_{sp}}
  \\
    &=&    \langle    \overrightarrow{{G}_1 },  \overrightarrow{{F_1}}\rangle _{H_{tr}}+
    \langle    \overrightarrow{{G}_2 },  \overrightarrow{{F_2}}\rangle _{H_{tr}} +
    \langle \Phi   \overrightarrow{{G}_1 }, \Phi\overrightarrow{{F_2}}\rangle_{\mathcal H_{sp}}  +
    {\overline {
    \langle \Phi   \overrightarrow{{F_1}}, \Phi\overrightarrow{G_2}\rangle}_{\mathcal H_{sp}}}.
   \end{eqnarray*}
	Therefore, {to show that
	\be\label{eq:isom}
	    \langle \Phi (  \overrightarrow{{G}_1 }+\overrightarrow{{G}_2 }),
   \Phi (  \overrightarrow{{F_1}}+\overrightarrow{{F_2} } )
   \rangle_{\mathcal H_{sp}} =   \langle  \overrightarrow{{G}_1 }+\overrightarrow{{G}_2 },
    \overrightarrow{{F_1}}+\overrightarrow{{F_2} } 
   \rangle_{\mathcal H_{tr}},
	\ee}
	it is enough to check that $\langle \Phi  \overrightarrow{{G}_1 }, \Phi  \overrightarrow{{F_2}} \rangle_{\mathcal H_{sp}}=  
  \langle  \overrightarrow{{G}_1 },   \overrightarrow{{F_2}} \rangle_{\mathcal H_{tr}}.$
   We have
    \begin{eqnarray*}
  \langle \Phi  \overrightarrow{{G}_1 }, \Phi  \overrightarrow{{F_2}} \rangle_{\mathcal H_{sp}}
  &=&   \left \langle   \left (  \begin{array}{cc}  I_E &  S^*(k) \\ S(k) & I_{E_*}
  \end{array}
  \right )
   \left (  \begin{array}{c}  \dfrac i { \sqrt{ 2 \pi}}   \dfrac e { k-\lambda} \\ - \dfrac i { \sqrt{ 2 \pi}}   \dfrac {S(\lambda) e}{k-\lambda}\end{array} \right ),
  {  \left (  \begin{array}{l}  
   \hat w_+ (k)+  \dfrac i { \sqrt{ 2 \pi}} \dfrac {S^*(\overline   \nu) f_*}{k-\nu} \\
  \hat w_- (k)-  \dfrac i { \sqrt{ 2 \pi}} \dfrac {f_*}  {k-\nu}   \end{array} \right )}
   \right \rangle_{ L_2 (\Rr,E \oplus E_*)}\\
	&=&
   \left \langle  
   \left (  \begin{array}{l}  
  \dfrac i { \sqrt{ 2 \pi} (k-\lambda)}
  ( I_E -  S^*(k)S(\lambda) ) e \\ \dfrac i { \sqrt{ 2 \pi}( k-\lambda)}( S(k) - S(\lambda) ) e
  \end{array}
  \right ),
   {  \left (  \begin{array}{l}  
   \hat w_+ (k)+  \dfrac i { \sqrt{ 2 \pi}} \dfrac {S^*(\overline   \nu) f_*}{k-\nu} \\
  \hat w_- (k)-  \dfrac i { \sqrt{ 2 \pi}} \dfrac {f_*}  {k-\nu}   \end{array} \right )}
   \right \rangle_{ L_2 (\Rr,E \oplus E_*)}.
    \end{eqnarray*}  
   Since $ \dfrac 1 { k-\lambda}  ( I_E-S^*(k) S(\lambda))e \in H_2^-(E)$ for $\lambda\in\C_+$, {while
     $ \hat w_+(k) + \dfrac i { \sqrt{ 2 \pi}}
    \dfrac {S^*(\overline \nu)f_*} { k-\nu}
    \in H_2^+(E)$ for $\nu\in\C_-$,} and
    $  \dfrac1 { k-\lambda}  (S(k)-S(\lambda))e  \in H_2^+(E_*)$ for  $\lambda\in\C_+$, while {$\hat  w_-(k) \in H_2^-(E_*)$},
		{many terms in the scalar product vanish and we are left with}
		    \begin{eqnarray*}
  \langle \Phi  \overrightarrow{{G}_1 }, \Phi  \overrightarrow{{F_2}} \rangle_{\mathcal H_{sp}}
  & =& \llangle
  \dfrac i { \sqrt{ 2 \pi}}  \dfrac { 1} { k-\lambda} (S(k)-S(\lambda))e,
  -\dfrac i { \sqrt{ 2 \pi}}
  {\dfrac { f_*} { k-\nu}}\rrangle_{L_2(\Rr,E_*)}.
    \end{eqnarray*}

    Finally, calculating the residue at point {$k=\overline \nu \in \mathbb C_+$},  we get
    \begin{eqnarray*}
   \langle \Phi \overrightarrow{{G}_1 },
    \Phi  \overrightarrow{{F_2}}
    \rangle_{H_{sp}}&=& -\dfrac 1{2 \pi} \llangle \dfrac{  (S(k)-S(\lambda))e}{k-\lambda}, {\dfrac{f_*}{k-\nu}}\rrangle_{L_2 (\Rr, E_*)} \ =\   -\dfrac 1{2 \pi}  \int_\Rr dk\dfrac{  \llangle ( S(k)-S(\lambda))e, {f_*}\rrangle_{E_*}}{(k-\lambda) (k-\overline \nu)}\\
    &
   = & \left ( -\dfrac{2\pi i} {2\pi }\right ) { \dfrac{\langle (    S(  \overline \nu)-S(\lambda))e, f_* \rangle_{E_*}}{(\overline \nu-\lambda)}}\ =\ -i   \left \langle {  \left ( \dfrac{  S(\overline \nu)-S(\lambda)}{\overline \nu-\lambda} \right ) e, f_*}\right \rangle_{E_*}.
 \end{eqnarray*}  
 By {\eqref{eq:Sdiff}, for $\lambda \in \C_+$ and $\nu\in \mathbb{C}_-$, we have}
$$ {\dfrac{ S(\overline \nu ) -S(\lambda)}{\overline \nu -\lambda}=i (\Gamma_*(A^*-\overline \nu)^{-1} )(  \Gamma (A-\overline \lambda))^{-1})^*.}$$
 Hence
 \bea
\langle \Phi \overrightarrow{{G}_1 }, \Phi \overrightarrow{{F_2}}\rangle_{\mathcal H_{sp}}&=&\langle (\Gamma_* (A^*-\overline \nu)^{-1}) (\Gamma (A-\overline \lambda))^{-1})^*e,f_*\rangle_{E_*}\\
&=&\langle (\Gamma (A-\overline \lambda )^{-1})^*e,( (\Gamma_* (A^*-\overline \nu))^{-1})^*f_*\rangle_{H} =\langle  \overrightarrow{{G}_1 },   \overrightarrow{{F_2}}\rangle_{\mathcal H_{tr}},
 \eea
{ as required. This completes the proof of the lemma.}
  \end{proof}

 \underline{\textbf{Step 4:}} $\Phi:\cH_{tr}\ominus (0,H_{sa},0) \to \mathcal H_{sp}$ is surjective.
 
 {We have shown that the map $\Phi$ admits a unique extension as an isometric operator to} the closure of all test vectors.
 According to  Lemma \ref{lem4.1}, the  closure coincides with $\mathcal H_{tr}\ominus(0,H_{sa},0)$,  where $\mathcal H_{sa}$  is the selfadjoint subspace of $A$ in the Langer decomposition.  We will use the same letter  $\Phi$  for the isometric extension  of $\Phi$ {to} $\mathcal H_{tr}\ominus (0,H_{sa},0)$. 
\begin{lemma}\label{lem4.5}
 The map $\Phi$,  { defined on test vectors in \eqref{em1} and \eqref{em2} and extended by linearity and continuity to} $\mathcal H_{tr} \ominus (0,H_{sa},0)$, has the property that $\overline{\Ran\Phi}=\mathcal H_{sp}$.
 \end{lemma}
 \begin{proof}
Consider a vector $\left  (   \begin{array}{c}  \tilde g \\ g \end{array}\right ) \in \mathcal  H_{sp}$
such that  $\left  (   \begin{array}{c}  \tilde g \\ g \end{array}\right ) $ is orthogonal to  $\Phi  \left  (   \begin{array}{c}  v_-\\0\\ v_+ \end{array}\right )$  for all $v_-\in L_2 (\Rr_-, E_*), v_+ \in L_2 (\Rr_+, E)$.
 Clearly the vector $ \left  (   \begin{array}{c}  v_-\\0\\v_+ \end{array}\right ) \in D (\Phi)={\mathcal H_{tr}} \ominus (0,H_{sa},0)$.  Moreover {it   is simultaneously a  test vector of the first and second type, with $e=0$ or $e_*=0$,}  and therefore its  image under $\Phi$    is easy to calculate. We have, due to $\hat v_+$ and $\hat v_-$ being $L_2$-functions,
 $$\left  (   \begin{array}{l}  \tilde g \\ g \end{array}\right ) \perp
 \Phi  \left  (   \begin{array}{c}  v_-\\0\\ v_+ \end{array}\right )=   \left  (   \begin{array}{l}  \hat v_+(k) \\ \hat v_-(k) \end{array}\right )\Leftrightarrow
 \left \langle    \left  (   \begin{array}{l}  \tilde g + S^*g\\ S \tilde g + g  \end{array}\right ),  \left  (   \begin{array}{l}  \hat v_+\\  \hat v_-   \end{array}\right ) \right \rangle_{  L_2  (\Rr; E\oplus E_*)}=0, $$
 i.e.~$\tilde g + S^* g\perp \hat v_+$  and  $S\tilde g +  g\perp \hat v_-$.
 Since, by  the Paley-Wiener Theorem  (see {\cite{Koosis}}), $\hat v_{\pm}$ { run  over the whole  of the spaces  $H_2^+ (E)$, $H_2^-(E_*)$, respectively,  and }   $ \tilde g +S^* g \in L_2(\Rr, E), \tilde Sg + g \in L_2(\Rr, E_*)$  we must have
 $    \tilde g + S^* g\in H_2^-( E)$  and  $   S  \tilde g +  g\in H_2^+( E_*)$.

 Additionally,
 $   \left  (   \begin{array}{c}  \tilde g  \\   g  \end{array}\right ) \perp \Phi   \left  (   \begin{array}{c}  0   \\  (\Gamma_* (A^*-\overline \mu)^{-1})^*e_*\\  0
  \end{array}\right )
 $  for all  $\mu \in \mathbb C_-$ and $e_*\in  E_*$.
 According to {\eqref{em2}},  we have
 $$\Phi   \left  (   \begin{array}{c}   0  \\  (\Gamma_* (A^*-\overline \mu)^{-1})^*e_*\\  0
  \end{array}\right ) =
   \left  (   \begin{array}{c}   \dfrac{ iS^* (\overline \mu ) e_*}{\sqrt{ 2 \pi}(k-\mu)} \\  - \dfrac{   ie_*}{\sqrt{ 2 \pi}(k-\mu)}
  \end{array}\right )
  $$
  and therefore by our assumption
  $$0=\left \langle \tilde g + S^*g,   \dfrac{ S^* (\overline \mu ) e_*}{k-\mu} \right \rangle_{L_2(\Rr,E)}-
  \left \langle S\tilde g + g, \dfrac{   e_*}{k-\mu}\right \rangle_{L_2(\Rr,E_*)}.
  $$
  {The first term is equal to $0$, since $\dfrac{ S^* (\overline \mu ) e_*}{k-\mu}\in H_2^+(E)$ and we have already seen that
		$g + S^*g\in  H_2^-(E)$}.  Hence
  $$  \left \langle S\tilde g + g, \dfrac{   e_*}{k-\mu}\right \rangle_{L_2(\Rr,E_*)} =0\quad\hbox{ for all } e_*\in E_* \hbox{ and } \mu \in \C_-,$$
  i.e. $$0= \dfrac1{2\pi i} \llangle S\tilde g + g, \dfrac{   e_*}{k-\mu} \rrangle_{L_2(\Rr,E_*)} =\dfrac1{2\pi i}  \int_\R \dfrac{  \langle (S\tilde g + g)(k),e_*\rangle_{E_*}}{k-\overline \mu}\ dk .$$
 The last equality means the Riesz projection $P_+$ onto $H_2^+$  of the scalar function $\langle  (S\tilde g + g)(k),e_*)_{E_*}  \in  {H_2}^+ $  is  equal  to  $ 0$, { and hence 
   $\langle  (S\tilde g + g)(k),e_*) \rangle _{E_*} =0 $ for a.e.~$k\in\R$ and  for any fixed $e_*\in E$.}  Choosing a countable orthonormal basis in $E_*$ as vectors $e_*$,  we have $(S\tilde g + g)(k)=0 $ for a.e.~$k\in \Rr$,  which means $S\tilde g + g=0$.
 
  Similarly, the other  condition
  $$\left  (   \begin{array}{l}  \tilde g \\ g \end{array}\right )\perp
  \Phi    \left  (   \begin{array}{c}  0   \\  (\Gamma(A-\overline \lambda)^{-1})^*e\\  0
  \end{array}\right )
  = \dfrac i { \sqrt{ 2 \pi}}  \left (   \begin{array}{c}   \dfrac{e}{k-\lambda} \\ -\dfrac{S(\lambda)e}{k-\lambda}
  \end{array}\right ), \lambda \in \C_+, e \in E,
  $$
  means that
  $$ \left  (   \begin{array}{c}   \tilde g + S^* g  \\  S\tilde g + g  
  \end{array}\right ) \perp
  \left  (   \begin{array}{c}   \dfrac e { k-\lambda} \\  -\dfrac{  S(\lambda) e}{k-\lambda}  \end{array}\right )$$
  in $L_2 (\Rr, E\oplus E_*)$.  
  Since $S\tilde g + g \in H_2^{+}(E_*)$ and $
  \dfrac{S(\lambda)e}{k-\lambda}\in H_2^-(E_*)$  for all $\lambda\in \mathbb C^{+}$ this orthogonality  condition can be written as
  $$ \llangle  \tilde g + S^*g, \dfrac e {k-\lambda} \rrangle _{L_2{(\Rr,E)}} =0, \quad \hbox{ for all } \quad e\in E, \lambda \in \mathbb C_+,$$
  i.e.  $P_- \langle \tilde g + S^* g,e\rangle_E=0$.  
  Since  $\langle \tilde g + S^* g,e\rangle)_E\in H_2^-$, we also have
  $\langle \tilde g + S^* g,e\rangle_E=0$ for all $e\in E$ or    $ \tilde g + S^* g=0 $, as in the previous case.
 In summary this yields
  $ (S\tilde g + g)=0$  and $ (\tilde g + S^* g)=0$  which, { by Lemma \ref{lem:uniquedet}}, gives
 $$  \left  (   \begin{array}{l}  \tilde g  \\   g  \end{array}\right )=0,$$ 
{which proves the result.}
\end{proof}

  Now we are ready to prove our previous claim concerning  the validity of using a single symbol  $\Phi$ for both maps of test vectors  acting separately on functions of the first and second kind.
 \begin{Lemma}\label{lem3.12}
 Let  two linear sets in $\mathcal H_{tr}$ be: 
$$\mathcal L_1 :=\Span  \left \{   \left (  \begin{array}{c}    v_- \\ ( \Gamma (A-\overline \lambda)^{-1})^*e \\v_+
  \end{array}
 \right ) \Bigg\vert v_-\in L_2 (\Rr_-, E_*), v_+\in   L_2 (\Rr_+, E),
 e \in E, \lambda \in \C_+ \right \}
 $$
 and
 $$\mathcal L_2 :=\Span  \left \{   \left (  \begin{array}{c}    v_- \\ ( \Gamma_* (A^*-\overline \mu)^{-1})^*e_* \\v_+
  \end{array}
 \right )  \Bigg\vert  v_-\in L_2 (\Rr_-, E_*), v_+\in   L_2 (\Rr_+, E),
 e_* \in E_*, \mu  \in \C_- \right \} ,
 $$
then 
$$ \overline{ \mathcal  { L }_1+  \mathcal  {L }_2
}  =\mathcal H_{tr}  \ominus (0,H_{sa},0)
$$
 and for any  vector $ \overrightarrow  F \in \overline {\mathcal  L}_1 \cap   \overline{  \mathcal  L}_2 $  we have $\Phi_1     \overrightarrow{F }
=\Phi_2     \overrightarrow{F}
 $ where
$\Phi_1$  
 is a map defined  on $\overline {\mathcal L }_1$ by (\ref{em1})  and    
 $\Phi_2$ is the map defined on
 $\overline {\mathcal L}_ 2$     by (\ref{em2})
 after taking the  closure of the isometric operators $\Phi_j : \mathcal L_j\to \mathcal H_{sp}$. $j=1,2$.
 \end{Lemma}
 \begin{proof}
That
$$ \overline{ \mathcal  { L }_1+  \mathcal  {L }_2
}  =\mathcal H_{tr}  \ominus (0,H_{sa},0)
$$ follows from Lemma \ref{lem4.1}.

 Let $ \overrightarrow{F}\in \overline {\mathcal L}_1 \cap  \overline {\mathcal L}_2$.
Then, using the isometry property of both maps $\Phi$ for test vectors of both first and second type, we may extend the identity \eqref{eq:isom} to the closure of both types of vectors, $\overline {\mathcal L}_1$ and $\overline {\mathcal L}_2$. Let 
$ \overrightarrow{F}=  \overrightarrow{F}_1=\overrightarrow{F}_2$ with $\overrightarrow{F}_1\in \overline {\mathcal L}_1$ and $\overrightarrow{F}_2 \cap  \overline {\mathcal L}_2$. Then for arbitrary $\overrightarrow{G}_1\in \overline {\mathcal L}_1$ and 
$\overrightarrow{G}_2\in \overline {\mathcal L}_2$
 $$ \langle \Phi_1  \overrightarrow{F}_1
 -\Phi_2  \overrightarrow{F}_2,  \Phi_1
  \overrightarrow{G}_1+ \Phi_2
  \overrightarrow{G}_2 \rangle_{\mathcal H_{sp}}=
  (  \overrightarrow{F}_1
 -   \overrightarrow{F}_2,    
  \overrightarrow{G}_1 +  
  \overrightarrow{G}_2 )_{\mathcal H_{tr}}=  0.
	$$
 
 Since by Lemma \ref{lem3.12} the set of images of the test vectors  from $\Span\{\mathcal L_1,\mathcal L_2\}$  is dense in $\mathcal H_{sp}$, we have
 $$ \Phi_1 \overrightarrow{F}-\Phi_2 \overrightarrow{F}=0.$$
 \end{proof}

{ \underline{\textbf{Step 5:}} The intertwining  identity}
 
  {So far we have seen that} Theorem \ref{thm4.5} delivers an isometric linear map
    $\Phi$  of   $\mathcal  H _{tr}\ominus \{  O \oplus H_{sa} \oplus O\}$ onto $ \mathcal H_{sp} =\left (  L_2 (\Rr; E\oplus E_*;   \left  ( \begin{array}{cc}   I_E & S^*(k)   \\  S(k) & I_{E_*} \end{array}\right )dk \right ) $
  satisfying conditions (\ref{eq10}).  The formula for $\Phi$ is   explicit on  the set of  special test vectors of    both the first and second kind.  {It remains to show that the transform  $\Phi$ constructed above gives the spectral representation of the minimal selfadjoint dilation $\mathcal L$ of the of the
completely  non-selfadjoint part of the maximally dissipative operator $A$.}

  \begin{Lemma}\label{lem:intertwining}
  (The intertwining  identity)
 {We have that}
  $\Phi ({\mathcal L_{tr}} -\lambda)^{-1}=(k-\lambda)^{-1} \Phi$  on $\mathcal H\ominus  \{ O\oplus H_{sa}\oplus O\}$ for all $\lambda\in \mathbb C\backslash \Rr$.
	 \end{Lemma}

  \begin{proof}
  Let us assume w.l.o.g. that $H_{sa}=O$, i.e.~$A$ is a completely non-selfadjoint operator. Then, { $\Phi$ being surjective and isometric, we have }
  $\Phi^*\Phi=I_{\mathcal H_{tr}}$ and $ \Phi\Phi^*=I_{\mathcal H_{sp}}$. We will calculate  the resolvent $(\mathcal L_{tr} -\lambda)^{-1}$
  on the whole space  $\mathcal H_{tr}$.
 
   The equality
  \be\label{eq:Phi}\Phi (\mathcal L_{tr}-\lambda)^{-1}=(k-\lambda)^{-1}\Phi \ee
  is equivalent  to
  \begin{equation}
  \mathcal F_\pm (\mathcal L_{tr} -\lambda)^{-1} = (k-\lambda)^{-1} \mathcal F_\pm, \label{e16}
  \end{equation}
  for both signs simultaneously.  Indeed, we have   checked in Lemma \ref{lem4.2} and Lemma \ref{lem4.3}   that
  our map $\Phi$ satisfies the condition:
  \begin{equation}
   \left  ( \begin{array}{cc}   I_E & S^*  \\  S & I_{E_*} \end{array}\right )
  \Phi  \overrightarrow{F }=
     \left  ( \begin{array}{cc}  
      \mathcal F_+    \overrightarrow{F}  \\  
 \mathcal F _-    \overrightarrow{F}
      \end{array}\right ) \label{e17}
     \end{equation}
    on the test vectors $\overrightarrow{ F }$, {which generate a dense set in $\mathcal H_{tr}$, under the condition $H_{sa}=\{0\}$.}  Since $\norm{\Phi}=1$  and the map
     $${\overrightarrow{F} \mapsto \left  ( \begin{array}{cc}  
      \mathcal F_+    \overrightarrow{F}  \\  
 \mathcal F _-    \overrightarrow{F}
      \end{array}\right )}$$ { from the set of test vectors to $L_2(\Rr,E)\oplus L_2(\Rr,E_*)$}
     has  norm which obviously does not exceed $2$, we can extend the equality (\ref{e17}) to the whole space $\mathcal H_{tr}$ .

  We remind   the reader that   the   vector    
    $  \left  ( \begin{array}{c}   \tilde g   \\   g\end{array}\right ) \in \mathcal H_{sp}$ is equal to $0$ iff  
    $\left  ( \begin{array}{cc}   I_E & S^*   \\  S& I_{E_*} \end{array}\right ) \left  ( \begin{array}{c}   \tilde g  \\  g \end{array}\right )=0$,  so the equality (\ref{e16}) is equivalent (after multiplication on the left by the matrix function of $k\in \Rr$\,\,  $\left  ( \begin{array}{cc}   I_E & S^* (k)  \\  S(k)& I_{E_*} \end{array}\right )$ and $(k-\lambda)^{-1}$), to the condition (\ref{eq:Phi}) due to the commutation of the  two operations of multiplication by $(k-\lambda)^{-1}$ and by the matrix-function mentioned above.

    Consider the vector
    $\overrightarrow{\mathcal F }=   \left  ( \begin{array}{c}   v_-\\u  \\   v_+ \end{array}\right ) \in \mathcal H_{tr}$  and denote $  (\mathcal L_{tr} -\lambda)^{-1}   \left  ( \begin{array}{c}    v_-\\u      \\v_+ \end{array}\right )=:
     \left  ( \begin{array}{c}   \tilde v_-\\ \tilde u ,\\ \tilde   v_+ \end{array}\right )\in \mathcal H_{tr}$, $\lambda \not \in \Rr$.
  Then $(\tilde v_-,\tilde u ,\tilde v_+)\in \mathcal  D(\mathcal L_{tr})$   and,  by the  definition of the dilation,
    \begin{eqnarray}
     v_- = i\tilde v_-'-\lambda \tilde v_-, \label{eq:minus}  \\
     v_+ = i\tilde v_+'-\lambda \tilde v_+, \label{eq:plus}\\
    u = T_*\left  (   \begin{array}{c}
   \tilde v_- \\ \tilde u\\
   \tilde v_+  
  \end{array}
  \right )-\lambda \tilde u.
  \label{e18}
  \end{eqnarray}
  Inclusion of $(\tilde v_-, \tilde u, \tilde v_+)$ in $\mathcal D(\mathcal
     L_{tr})$ leads additionally to the following facts:
   $\tilde v_+\in H^1 (\Rr_+,E),
   \tilde v_-\in H^1 (\Rr_-,E_*)$ and
  \be\label{eq:dom} \left \{
  \begin{array}{c}
  \tilde u + (\Gamma_* (A^*+\mu)^{-1})^*\tilde v_-(0)\in \mathcal{D}(A),\mu\in \mathbb C_-,\\
  \tilde v_+(0)=S^*(-\mu) \tilde v_-(0) + i \Gamma (\tilde u + (\Gamma_*(A^*+\mu)^{-1})^*\tilde v_-(0)).
  \end{array} \right .
  \ee

  Now we need to prove that for any fixed $\lambda \in \mathbb C\backslash \Rr$
  \begin{equation}
  \mathcal F_\pm  
 \left (\begin{array}{c}
   \tilde v_- \\ \tilde u\\
   \tilde v_+  
  \end{array}
  \right )= (k-\lambda)^{-1}  \mathcal F_\pm
 \left  (\begin{array}{c}
    v_- \\  u\\
    v_+  
  \end{array}
  \right ).
  \label{e19}
  \end{equation}
  Using the Fourier transform for equation { 
  (\ref{eq:minus}),} we get, after extension of {$v_-, \tilde v_-$  by $0$ on the positive half-line}, that
  $$\hat v_-(\xi)=\dfrac i { \sqrt {2 \pi}} \int _{-\infty}^0 e^{i \xi t} \tilde v_-' (t) dt \ -\lambda
  \hat {  \tilde v}_-(\xi)=
  (\xi-\lambda) \hat {  \tilde v}_-(\xi)+\dfrac i { \sqrt { 2 \pi}} \tilde v_-(0).
  $$
  Similarly, {using \eqref{eq:plus}, we obtain}
  $$ \hat v_+(\xi)=  (\xi -\lambda) \hat {  \tilde v}_+(\xi)-\dfrac i { \sqrt { 2 \pi}} \tilde v_+(0).
  $$
    Using {\eqref{eq:Tstar}}  for the operator $T_*$ {in \eqref{e18},} we have, independently of  $\mu \in \mathbb \C_-$,
    $$u= A ({\tilde u+}(\Gamma_*(A^*+\mu)^{-1})^*\tilde v_- (0))+\overline \mu(\Gamma_*(A^* + \mu)^{-1})^* {\tilde v_- (0)}
-\lambda \tilde u).
$$
   If we fix the value of the parameter $\lambda\in \mathbb \C_-$  (it is enough to prove {(\ref{e15a}) for $\lambda$ in a    half-plane, as the result on the complementary half-plane follows immediately by taking adjoint operators in (\ref{e15a})}),  the convenient choice of $\mu$ is $\mu=-\overline \lambda \in \mathbb \C_-$.  Then
  $$
  u= (A-\lambda) [  \tilde
  u +(\Gamma_*(A^* -\overline \lambda  )^{-1})^*\tilde v_- (0)]\quad {\hbox{ or }\quad \tilde u = (A-\lambda)^{-1}u -(\Gamma _+(A^*-\overline \lambda)^{-1})^*\tilde v_-(0).}
$$  
Let us first consider the case $\mathcal F_+ $ in (\ref{e19}). {Inserting the expression for $\tilde u$ from above}, we now need to prove that for a.e. $k \in \mathbb R$
\begin{eqnarray*}
&& \lim_{{\eps \to + 0}}  (  -\dfrac 1 {  \sqrt{ 2 \pi}}) \Gamma (A-k + i \eps)^{-1} [  (A-\lambda)^{-1}u -(\Gamma _*(A^*-\overline \lambda)^{-1})^*\tilde v_-(0)]
+ S^*(k) \hat {\tilde v}_-(k)+ \hat {\tilde v}_+(k)\\
&&= (k-\lambda)^{-1} [
\lim_{{\eps \to + 0}}  (  -\dfrac 1 {  \sqrt{ 2 \pi}}) \Gamma (A-k + i \eps)^{-1}u + S^* (k) \hat v_- (k)+\hat v_+(k)].
\end{eqnarray*}
Using the Hilbert identity
$$ (A-k+i \eps )^{-1} (A-\lambda)^{-1}=
\dfrac  {  (A-\lambda)^{-1}-(A-k+i\eps)^{-1}}{\lambda -k + i \eps}
 $$
and substituting the explicit expressions for $\hat v_\pm (\xi)$ calculated above, we get   that the equality we have to prove can be reduced to the following:
\be\label{eq:3.22}\Gamma (A-\lambda)^{-1} u + (k-\lambda) \lim_{\eps \to +0}  \Gamma (A-k+ i \eps)^{-1}  ( \Gamma_*(A^*-\overline \lambda)^{-1})^* \tilde v_-(0) + i[  \tilde v_+(0)-S^*(k) \tilde v_-(0)]=0.
\ee

{Taking adjoints in \eqref{eq:Sdiff}, with $\mu=\overline \lambda$ and $\tilde\mu=\overline z$, both in $\C_+$ we have that
$$ (\lambda-z) \Gamma (A-\overline z)^{-1}  (\Gamma_* (A^*-\overline  \lambda)^{-1})^*=i  (S^*(\overline \lambda)-S^*(\overline z)).
$$}
To complete the proof for $\mathcal F_+$  we need to substitute $z:=k-i \eps\in \mathcal \C_-$  and take  the limit as $\eps \to +0$ in the strong topology of $E_*$ for a.e.~$k$.  Indeed, { following this procedure the proof of \eqref{eq:3.22} reduces to}
$$ 0=\Gamma (A-\lambda)
^{-1} u-i (S^*(\overline \lambda)-S^*(k)) \tilde v_-(0) + i[\tilde v_+(0) - S^*(k) \tilde v_-(0)]
 = \Gamma(A-\lambda)^{-1}u + i \tilde v_+(0) -i S^*(\overline \lambda) \tilde v_-(0).$$
 {Substituting $\mu=-\overline \lambda$ into \eqref{eq:dom}  and taking into account that  $(A-\lambda)^{-1}u=\tilde u +(\Gamma_*(A^*-\overline \lambda)^{-1})^*\tilde v_-(0)$,
 we see that the  expression vanishes, as desired}.
 The second equality
 $$
 \mathcal F_-
 \left (\begin{array}{c}
   \tilde v_- \\ \tilde u\\
   \tilde v_+  
  \end{array}
  \right )=(k-\lambda)^{-1}  
 \mathcal F_-   \left (\begin{array}{c}
    v_- \\   u\\
    v_+  
  \end{array}
  \right )
  $$
  admits a  similar proof. {Although $T=T_*$ on $\mathcal {D(L}_{tr})$, it is more convenient to use the operator $T$ from \eqref{eq:T}
 for the $\mathcal F_-$ case.}
   \end{proof}
	{This completes the proof of the main theorem.}
	\end{proof}
	
	 \begin{remark} 
  For minimal selfadjoint dilations   ${\mathcal L_{tr}} $ of a general maximally dissipative operator $A$ we have immediately  from {Lemma \ref{lem:intertwining}} and the Langer decomposition Theorem \ref{thm:langer}  that
  \begin{equation}
  (\mathcal L_{tr}-\lambda)^{-1}=\Phi^* (k-\lambda)^{-1} \Phi \oplus (A|_{H_{sa}}-\lambda)^{-1},\label{e15a}
   \end{equation}
  where $\mathcal L_{tr}$ is an operator in the translation form space $\mathcal H_{tr}$ and the orthogonal sum corresponds to the Langer decomposition
  $$
  \mathcal H_{tr}=\left( \mathcal H_{tr}\ominus (0,H_{sa},0)\right) \oplus (0,H_{sa},0)
  $$ and the selfadjoint operator $A|_{H_{sa}}$
  acts in $(0,H_{sa},0)$ as an operator in the second component.
   \end{remark}
	
	{From the theorem, we get the following corollary.}
	
   \begin{Corollary} \label{unlabCorp64}
   We have the dilation property 
   $$
   P_H(\mathcal L _{tr}-\lambda)^{-1}|_H=
   \left \{
   \begin{array}{cc}
   (A^*-\lambda)^{-1}, & \lambda\in\C_+,\\(A-\lambda)^{-1}, & \lambda\in\C_-  \end{array} \right .
   $$
   in $\mathcal H_{tr}$. If $A$ is completely non-selfadjoint, this can be transformed into the spectral form version
   $$
   P_H\Phi^*(k-\lambda)^{-1}\Phi |_H=
    \left \{
   \begin{array}{cc}
   (A^*-\lambda)^{-1}, & \lambda\in\C_+,\\(A-\lambda)^{-1}, & \lambda\in\C_- ,  \end{array} \right .
   $$
   or
   $$
   (\Phi P_H \Phi^*)(k-\lambda)^{-1}|_{{\Phi (H)}}=
    \left \{
   \begin{array}{cc}
  \Phi (A^*-\lambda)^{-1} \Phi^*|_{{\Phi(H)}},& \lambda\in\C_+,\\\Phi (A-\lambda)^{-1}\Phi^*|_{{\Phi (H)}} ,  &\lambda\in\C_-.   \end{array} \right . 
   $$
	\end{Corollary}

  {Set $K:= \Phi (0,H,0)\subset \mathcal H_{sp}$}. If $A$ is a completely non-selfadjoint  maximally dissipative operator, then
   $ \Phi P_H \Phi^*=P_K$ is the orthogonal projection on to the subspace $K$ in $\mathcal H_{sp}$  and
   $P_K\dfrac 1 { k-\lambda}$ is unitarily equivalent to the resolvent of $A$ (for $\lambda\in\C_-$) or $A^*$ (if $\lambda\in\C_+$).
   Explicit calculations  (see \cite{Pav75,Pav76}) give that
   $$
   P_K  \left  ( \begin{array}{c} \tilde g\\ g
 \end{array} \right )=
 \left  ( \begin{array}{c} \tilde g- P_+(\tilde g + S^* g)\\ g  - {P_-}(S\tilde g +  g)
 \end{array} \right ) \quad \hbox{ for }\quad \left  ( \begin{array}{c} \tilde g\\ g
 \end{array} \right ) \in \mathcal H_{sp},
 $$
 where $P_\pm$ are Riesz projections onto
 {$$H_2^+(E)\subset L_2(\Rr, E)\quad \hbox{ and }\quad H_2^-(E_*)\subset L_2(\Rr, E_*)
 $$
 respectively.}  The last formula is well-defined since $\tilde g + S^* g \in L_2(\Rr, E)$ and
 $S\tilde g +  g \in L_2(\Rr, E_*)$  for all $(\tilde g, g)\in \mathcal H_{sp}$.
 
 {Note that  the images of  Lax-Phillips's  incoming and outgoing channels  (subspaces) $\mathcal D_-=(  L_2(\Rr_-, E_*), 0,0  )$  and
 $\mathcal D_+=(0,0,L_2 (\Rr_+, E))$
 under $\Phi$ are}
 $${\Phi}\mathcal D_-= \left  ( \begin{array}{c} 0\\ H_2^-(E_*) \;\; \end{array} \right ),  \quad
 {\Phi}\mathcal D_+=  \left  ( \begin{array}{c} H_2^+(E)\\ 0 \end{array} \right )
 $$
 in a completely symmetric way.  This property of Pavlov's version of the functional  model is a very convenient  tool  in model calculations.
 Note that the minimality of the selfadjoint dilation  follows immediately from the minimality of the translation form of the dilation.  The last fact holds true in {both the completely non-selfadjoint  and the general maximally dissipative operator cases}.

 \section{Example: a limit-circle problem}\label{section:5}
 Explicit calculation of the ingredients appearing in the functional model, for concrete examples, can be quite non-trivial. In this section we 
 consider a one-dimensional Schr\"{o}dinger problem with one regular endpoint and one singular, limit-circle endpoint, and compute expressions
 for the characteristic function and two other operators appearing in the functional model. {We choose a limit-circle endpoint since
 this allows freedom of choice in the boundary conditions, and hence reveals the different explicit r\^{o}les of the boundary conditions and  of
 the imaginary part of the potential.}  Our calculations allow for a limit-circle-oscillatory endpoint,
 and hence for spectrum with real part unbounded below; {limit-circle non-oscillatory endpoints can be transformed to
 regular endpoints \cite{NiessenZettl92} and are therefore covered by our previous work \cite{BMNW20}.}
 As for all limit-circle problems, there is no essential spectrum.
 
 Consider the expression
\[ \ell u := -u'' + Q(x) u  \;\;\; x \in (0,1]; \]
here we suppose that $Q$ is real-valued, regular at $x=1$ and limit-circle at $x=0$. We choose a real-valued
basis $\{ c, s\}$ of the solution space of the equation $\ell u = 0$ determined by initial conditions 
$s(1) = 0$, $s'(1) = 1$, together with the Wronskian condition $sc'-s'c\equiv 1$. We associate 
with the expression $\ell$ an operator $L_B$ with domain
\[ D(L_B) = \{u\in L_2(0,1)\; | \; \ell u \in L_2(0,1), \;\; u(1) = 0,\;\; [u,s](0) + B [u,c](0) = 0\}. \]
Here the square bracket notation denotes the Wronskian, i.e.~$[u,s](x) = u(x)s'(x)-u'(x)s(x)$, and values
at $x=0$ are to be interpreted in terms of limits. $B\neq 0$ is a complex number; if $B$ is real then it
is well known that $L_B$ is self-adjoint.

Assume that $\lambda=0$ is not an eigenvalue of $L_B$. Then we may calculate the resolvent
of $L_B$ by the variation of parameters formula: $u = L_B^{-1}f$ if and only if
\begin{equation}\label{eq:1}u(x) = c(x) \int_x^1 s(t)f(t)dt + s(x) \int_0^x c(t) f(t)dt + \frac{1}{B} s(x) \int_0^1 s(t)f(t)dt. \end{equation}
It is then a simple calculation to show that
\begin{equation}\label{eq:2}u'(x) = c'(x) \int_x^1 s(t)f(t)dt + s'(x) \int_0^x c(t) f(t)dt + \frac{1}{B} s'(x) \int_0^1 s(t)f(t)dt. \end{equation}

Now we wish to examine conditions on $B$ to have a dissipative operator $L_B$. Evidently
\begin{equation}\label{eq:3} \langle L_B u,u \rangle = \lim_{x\searrow 0} \left[ u'(x)\overline{u(x)} + \int_x^1 \left( |u'(t)|^2 + Q(t)|u(t)|^2 dt\right) \right], 
\end{equation}
and so since $Q$ is real-valued it follows that $L_B$ is dissipative if and only if for all $u\in D(L_B)$
\[  \lim_{x\searrow 0} \Im(u'(x)\overline{u(x)}) \geq 0. \] 

In order to simplify the calculations slightly we observe that if we restrict our attention to functions $f$ which
vanish in a neighbourhood of $x=0$ then, since such $f$ are dense in $L_2(0,1)$, the resulting $u=L_B^{-1}f$
which we generate will form a core of $D(L_B)$. It is therefore sufficient to check dissipativity on such $u$.
If $x$ is sufficiently small to lie outside the support of $f$ then from (\ref{eq:1}) and (\ref{eq:2}),
\[ u(x) = \left(c(x) + \frac{1}{B} s(x)\right) \int_0^1 s(t)f(t)dt, \;\;\; u'(x) = \left(c'(x) + \frac{1}{B} s'(x)\right) \int_0^1 s(t)f(t)dt, \]
and thus, as $c$ and $s$ are real-valued,
\begin{equation}\label{eq:4} \Im(u'(x)\overline{u(x)}) = \Im\left(\frac{1}{B}\right) (s'(x)c(x)-s(x)c'(x))\left|\int_0^1 s(t)f(t)dt\right|^2 = -\Im\left(\frac{1}{B}\right)\left|\int_0^1 s(t)f(t)dt\right|^2, 
\end{equation}
where in the last step we have used the fact that $sc'-s'c\equiv 1$. Thus $L_B$ is dissipative if and only if $\Im(B)\geq 0.$ {We will assume $\Im(B)\geq 0$ from now on.}

We now cast this example into a boundary-triples framework \mm{\cite{BMNW08,DM91,DM92,GG91}}. Our maximal operator $L_{max}$ is given by the expression
\[ L_{max}u = \ell u; \;\;\; D(L_{max}) = \{u\in L_2(0,1)\; | \; \ell u \in L_2(0,1), \;\; u(1) = 0\}, \]
and we wish to compute $\langle L_{max}f,g\rangle - \langle f,L_{max}g\rangle$, for $f,g\in D(L_{max})$. Using the von Neumann decomposition,
together with the fact that $[s,c]=1$, we have, in a neighbourhood of $x=0$,
\[  f(x) = f_0(x) + [f,c](0)s(x) - [f,s](0)c(x), \;\;\; g(x) = g_0(x) + [g,c](0)s(x) - [g,s](0)c(x), \]
in which $f_0,g_0 \in D(L_{max}^*)=\{u\in L_2(0,1)\; | \; \ell u \in L_2(0,1), \;\; u(1) = 0, \; [u,c](0)=0, \; [u,s](0)=0\}$. Also, a straightforward calculation using integration by parts shows that
\[ \langle L_{max}f,g\rangle - \langle f,L_{max}g\rangle = -[f,\overline{g}](0). \] 
It then follows that
\[  \langle L_{max}f,g\rangle - \langle f,L_{max}g\rangle = -[f,\overline{g}](0) = [f,c](0)\overline{[g,s](0)} - [f,s](0)\overline{[g,c](0)}. \]
If we define boundary operators $\Gamma_0$, $\Gamma_1$ on $D(L_{max})$ by
\[ \Gamma_0 f = [f,c](0), \;\;\; \Gamma_1 f = [f,s](0), \]
then the fundamental boundary triple identity 
can be written in the usual form
\[  \langle L_{max}f,g\rangle - \langle f,L_{max}g\rangle = \Gamma_0 f \, \overline{\Gamma_1 g} - \Gamma_1 f \, \overline{\Gamma_0 g}. \]
The boundary condition associated with $D(L_B)$ is $\Gamma_1 u + B \Gamma_0 u = 0$. Using equations (\ref{eq:1}), (\ref{eq:2}) we see
that if $u = L_B^{-1}f \in D(L_B)$ then
\[ [u,c](x) = \int_0^x c(t)f(t)dt + \frac{1}{B}\int_0^1s(t)f(t)dt, \]
whence, taking the limit as $x\searrow0$,
\[ \Gamma_0 u = \frac{1}{B}\int_0^1s(t)f(t)dt. \]
Combining this with (\ref{eq:3}) and (\ref{eq:4}) we find that 
\[ \Im \langle L_Bu,u\rangle = \Im(B) |\Gamma_0 u|^2 = \left|\sqrt{\Im(B)}\Gamma_0 u\right|^2.\]

Let $V\in L_\infty(0,1)$ be an essentially bounded, non-negative function. We define an operator $A_B$ by
\[ A_B = L_B + i V; \;\;\; D(A_B) = D(L_B). \]
Then 
\begin{equation}\label{eq:5} 
\Im \langle A_Bu,u\rangle = \Im \langle L_Bu,u\rangle +  \langle Vu,u\rangle =  \left|\sqrt{\Im(B)}\Gamma_0 u\right|^2 +  \langle\sqrt{V}u,\sqrt{V}u\rangle
\end{equation}
If we define a map $\Gamma:D(A_B)\longrightarrow {\mathbb C}\oplus L_2(V^{-1}({\mathbb R}_+))$ by
\begin{equation}\label{eq:6}
\Gamma u = \left(\begin{array}{c} \sqrt{\Im B}\Gamma_0 u \\ \sqrt{V}u \end{array}\right) = \left(\begin{array}{c} \sqrt{\Im B}[u,c](0) \\ \sqrt{V}u \end{array}\right)
\end{equation}
then we have the Lagrange identity
\begin{equation}\label{eq:7} 
\Im\langle A_Bu,u\rangle = \left\langle\Gamma u, \Gamma u \right\rangle_{{\mathbb C}\oplus L_2(V^{-1}({\mathbb R}_+))}. 
\end{equation}
Note that for this example, we also have
\begin{equation}\label{eq:8} 
\Im\langle A_B^*u,u\rangle = -\left\langle\Gamma u, \Gamma u \right\rangle_{{\mathbb C}\oplus L_2(V^{-1}({\mathbb R}_+))}. 
\end{equation}
We are thus in the simple situation $E = E_*$ and $\Gamma_* = \Gamma$.

The characteristic function $S(z)$ is defined by 
\[ S(z)\Gamma u = \Gamma (A_B^*-z)^{-1}(A_B-z)u, \quad {z\in\C_+}. \]
We now calculate $S(z)$. The first step is to find an expression for $v:=(A_B^*-z)^{-1}(A_B-z)u$. To this end we introduce solutions
$\tilde{s}_z$ and $\tilde{\phi}_z$ of the formal adjoint equation
\[ -y'' + (Q-iV)y = z y, \]
determined by the conditions 
\begin{equation}\label{eq:ic}
 \tilde{s}_z(1) = 0, \; \tilde{s}_z'(1) = 1;  \;\;\; \Gamma_0\tilde{\phi}_z = -1, \; \Gamma_1 \tilde{\phi}_z = \overline{B}.
 \end{equation}
 The existence of $\tilde{s}_z$, which is an entire function of $z$, is immediate from standard results on regular initial value
problems. The existence of an entire $\tilde{\phi}_z$ is less obvious, but may be proved by using a variation-of-parameters
argument. From (\ref{eq:ic}), $\tilde{\phi}_z$ satisfies the left-hand boundary condition associated with $A_B^*$, viz.
 \begin{equation} \Gamma_1 \tilde{\phi}_z + \overline{B} \Gamma_0 \tilde{\phi}_z = 0. \label{eq:mmbc17}\end{equation}
Moreover, 
\begin{equation}\label{eq:ic2} \tilde{\phi}_z = -\overline{B}c - s + g_z, \end{equation}
in which $g_z$ is a function such that $\Gamma_0 g_z = 0 = \Gamma_1 g_z$.

We also define the function $\tilde{M}(z)$ by
\begin{equation}\label{Mtdef} \Gamma_1 \tilde{s}_z = \tilde{M}(z)  \Gamma_0 \tilde{s}_z; \end{equation}
this means
\[ \tilde{M}(z) = \lim_{x\searrow 0} \frac{\tilde{s}_z(x)s'(x) - \tilde{s}_z'(x)s(x)}{\tilde{s}_z(x)c'(x) - \tilde{s}_z'(x)c(x)}. \]
{Note that the denominator does not vanish, as $\Gamma_0 \tilde{s}_z=0$ would imply that $z\in\C_+$ is an eigenvalue of the anti-dissipative operator $A_\infty^*$ with eigenfunction $\tilde{s}_z$, which is impossible.}
The equation $v=(A_B^*-z)^{-1}(A_B-z)u$ is equivalent to $(A_B^*-z)v=(A_B-z)u$, which means that
\[ -(v-u)'' + (Q-iV)(v-u) - z (v-u) = 2iVu. \]
We have $v(1)=0=u(1)$ and so variation of parameters yields, for some constant $a\in\mathbb C$,
\begin{equation}\label{eq:9a} v(x) = u(x) + \frac{\tilde{\phi}_z(x) \int_x^1 \tilde{s}_z(t)2iV(t)u(t)dt + \tilde{s}_z(x)\int_0^x \tilde{\phi}_z(t)2iV(t)u(t)dt}{[\tilde{\phi}_z,\tilde{s}_z]} + a \tilde{s}_z(x). \end{equation}
The value of $a$ is determined by imposing the condition $v\in D(A_B^*)$, which means
\[ \Gamma_1 v + \overline{B}\Gamma_0 v = 0. \]
Before doing this, however, we manipulate the denominator $[\tilde{\phi}_z,\tilde{s}_z]$ appearing in (\ref{eq:9a}). In view of (\ref{eq:ic2}) we have
\[ [\tilde{\phi}_z,\tilde{s}_z] = [-\overline{B}c-s,\tilde{s}_z] = \overline{B}\Gamma_0 \tilde{s}_z + \Gamma_1 \tilde{s}_z = (\overline{B}+\tilde{M}(z))\Gamma_0\tilde{s}_z. \]
Thus eqn. (\ref{eq:9a}) becomes
\begin{equation}\label{eq:9} v(x) = u(x) + \frac{\tilde{\phi}_z(x) \int_x^1 \tilde{s}_z(t)2iV(t)u(t)dt + \tilde{s}_z(x)\int_0^x \tilde{\phi}_z(t)2iV(t)u(t)dt}{(\overline{B}+\tilde{M}(z))\Gamma_0\tilde{s}_z } + a \tilde{s}_z(x). \end{equation}
From (\ref{eq:9}), bearing in mind that $(\Gamma_1+\overline{B}\Gamma_0)\tilde{\phi}_z=0$ and $(\Gamma_1+B\Gamma_0)u=0$, it follows that
\[ \Gamma_1v + \overline{B}\Gamma_0v  = -2i\Im(B)\Gamma_0 u  + a(\tilde{M}(z)+\overline{B})\Gamma_0 \tilde{s}_z,\]
whence, since $\Gamma_1v + \overline{B}\Gamma_0v=0$, we have
\[ a = \frac{2i\Im(B)\Gamma_0 u}{(\overline{B}+\tilde{M}(z))\Gamma_0 \tilde{s}_z}, \]
and
\begin{equation}\label{eq:9c} 
v(x) = u(x) + \frac{\tilde{\phi}_z(x) \int_x^1 \tilde{s}_z(t)2iV(t)u(t)dt + \tilde{s}_z(x)\int_0^x \tilde{\phi}_z(t)2iV(t)u(t)dt}{(\overline{B}+\tilde{M}(z))\Gamma_0\tilde{s}_z } 
+\frac{2i\Im(B)\Gamma_0 u}{(\overline{B}+\tilde{M}(z))\Gamma_0 \tilde{s}_z}\tilde{s}_z(x). 
\end{equation}
In particular, recalling that $\Gamma_0\tilde{\phi}_z = -1$, see eqn. (\ref{eq:ic}), it now follows that
\[ \Gamma_0 v = \left\{\frac{B+\tilde{M}(z)}{\overline{B}+\tilde{M}(z)}\right\}\Gamma_0 u - \frac{1}{(\overline{B}+\tilde{M}(z))\Gamma_0\tilde{s}_z }\int_0^1 \tilde{s}_z2iVu. \]
Observing that $Vu = \sqrt{V}\sqrt{V}u$, the characteristic function can be written as a $2 \times 2$ block operator matrix,
\[ S(z) = \left(\begin{array}{cc} S_{11}(z) & S_{12}(z) \\ S_{21}(z) & S_{22}(z) \end{array}\right), \]
in which
\begin{equation}\label{eq:s11s12} \left.
\begin{array}{c}S_{11}(z) = {\displaystyle \left\{\frac{B+\tilde{M}(z)}{\overline{B}+\tilde{M}(z)}\right\} = {1 + \frac{2i\Im(B)}{\overline{B}+\tilde{M}(z)},} \;\;\;\;
 S_{12}(z)\bullet  = \frac{-2i\sqrt{\Im B}}{(\overline{B}+\tilde{M}(z))\Gamma_0\tilde{s}_z }\int_0^1 \tilde{s}_z\sqrt{V}\bullet,} \\
 \\
{\displaystyle S_{21}(z) = \left\{ \frac{2i\sqrt{\Im B}\sqrt{V} \tilde{s}_z}{(\overline{B}+\tilde{M}(z))\Gamma_0\tilde{s}_z}\right\},  }\\
 \\
{\displaystyle S_{22}(z) \bullet = I \bullet 
+ \frac{2i\sqrt{V}}{(\overline{B}+\tilde{M}(z))\Gamma_0\tilde{s}_z}\left\{\tilde{\phi}_z\int_x^1 \tilde{s}_z \sqrt{V} \bullet + \tilde{s}_z\int_0^x \tilde{\phi}_z\sqrt{V}\bullet\right\} = {I \bullet + 2i \sqrt{V(x)}\int_{0}^1 G(x,t;z)\sqrt{V(t)}\bullet(t)dt,} }
\end{array}\right\}
\end{equation}
in which
{
\begin{equation}\label{GVerdi}
G(x,t;z) = \frac{\tilde{\phi}_z(\min(x,t))\,\tilde{s}_z(\max(x,t))}{(\overline{B}+\tilde{M}(z))\Gamma_0\tilde{s}_z}.
\end{equation}
}
\begin{Remark}
Since $\tilde{M}(z)\Gamma_0 \tilde{s}_z = \Gamma_1 \tilde{s}_z$, see (\ref{Mtdef}), the condition $(\overline{B}+\tilde{M}(z))\Gamma_0\tilde{s}_z=0$ 
is equivalent to $\overline{B}\Gamma_0 \tilde{s}_z + \Gamma_1\tilde{s}_z = 0$, which happens precisely when $\tilde{s}_z$ is an eigenfunction of
$A_B^*$. Since the singular point associated with the differential expression for $L_B$ is of limit circle type, $L_B$ has empty essential spectrum. The same is true of $A_B$ and $A_B^*$ since these are relatively compact perturbations of $L_B$. The singularities of 
$S(z)$ are therefore precisely the eigenvalues of $A_B^*$. If $\Im B >0$ then these lie strictly in the lower half-plane.
\end{Remark}

The other two main ingredients which appear in the functional model, and for which explicit expressions can be found in terms of solutions of initial
value problems and $M$-functions, are the operators $\Gamma(A_B-z)^{-1}$ and $\Gamma_*(A_B^*-z)^{-1}$. Calculating these quantities
is not more difficult than calculating the characteristic function $S$ itself. We illustrate this by obtaining an expression for $\Gamma(A_B-z)^{-1}$. The ingredients required are
the solutions ${s}_z$ and ${\phi}_z$ of the equation
\[ -y'' + (Q+iV)y = z y, \quad {z\in\C_-}, \]
determined by the conditions 
\begin{equation}\label{eq:icb}
 {s}_z(1) = 0, \; {s}_z'(1) = 1;  \;\;\; \Gamma_0{\phi}_z = -1, \; \Gamma_1 {\phi}_z = {B}. 
 \end{equation}
The conditions on ${\phi}_z$ ensure that 
\begin{equation}\label{eq:ic2b} {\phi}_z = -{B}c - s + g_z, \end{equation}
in which $\Gamma_0 g_z = 0 = \Gamma_1 g_z$.

We also define the function ${M}(z)$ by
\[ \Gamma_1 {s}_z = {M}(z)  \Gamma_0 {s_z}; \]
this means
\[ {M}(z) = \lim_{x\searrow 0} \frac{{s}_z(x)s'(x) - {s}_z'(x)s(x)}{{s}_z(x)c'(x) - {s}_z'(x)c(x)}. \]
A calculation similar to (but simpler than) the one which leads to eqn. (\ref{eq:9}) shows that the resolvent $(A_B-z)^{-1}$ is given by
\begin{equation}\label{eq:resolv} ((A_B-z)^{-1}f)(x) = \frac{{\phi}_z(x) \int_x^1 {s}_z(t)f(t)dt + {s}_z(x)\int_0^x {\phi}_z(t)f(t)dt}{(B+{M}(z))\Gamma_0s_z },  \end{equation}
and so, remembering that $\Gamma_0\phi_z = -1$, 
\[ \Gamma_0(A_B-z)^{-1}f = \frac{-1}{(B+{M}(z))\Gamma_0 s_z}\int_0^1 {s}_z(t)f(t)dt. \]
\begin{equation}\label{mm:last1} (\Gamma (A_B-z)^{-1}f)(x) = \left(\begin{array}{c}  {\displaystyle \frac{-\sqrt{\Im B} }{(B+{M}(z))\Gamma_0 s_z}\int_0^1 {s}_z(t)f(t)dt  } \\ \\
\sqrt{V(x)}\;{\displaystyle \frac{{\phi}_z(x) \int_x^1 {s}_z(t)f(t)dt + {s}_z(x)\int_0^x {\phi}_z(t)f(t)dt}{(B+{M}(z))\Gamma_0s_z }}
\end{array}\right). \end{equation}
For reference, we mention the corresponding expression for $\Gamma (A_B^*-z)^{-1}$, viz.
\[ (\Gamma (A_B^*-z)^{-1}f)(x) = \left(\begin{array}{c}  {\displaystyle \frac{-\sqrt{\Im B} }{(\overline{B}+\tilde{M}(z))\Gamma_0 \tilde{s}_z}\int_0^1 \tilde{s}_z(t)f(t)dt  } \\ \\
\sqrt{V(x)}\;{\displaystyle \frac{\tilde{\phi}_z(x) \int_x^1 \tilde{s}_z(t)f(t)dt + \tilde{s}_z(x)\int_0^x \tilde{\phi}_z(t)f(t)dt}{(\overline{B}+{\tilde{M}}(z))\Gamma_0\tilde{s}_z }}
\end{array}\right). \]

The expression for the map $\Phi$ given in Lemma \ref{lem4.2} shows how it acts upon vectors whose middle component
is of the form $(\Gamma(A_B-\overline{\lambda_0})^{-1})^*e$ for some $e\in E$, while the corresponding formula in Lemma 
\ref{lem4.3} gives the action of $\Phi$ upon vectors with middle component of the form $(\Gamma_*(A_B^*-\overline{\mu_0})^{-1})^*e_*$ for some $e_*\in E_*$. It is therefore useful to have expressions in our current example
for the inverses $((\Gamma(A_B-\overline{\lambda_0})^{-1})^*)^{-1}$ and $((\Gamma_*(A_B^*-\overline{\mu_0})^{-1})^*)^{-1}$, 
which we now obtain.

{Firstly, we may write (\ref{mm:last1}) in the form
 \[ (\Gamma (A_B-z)^{-1}f)(x) =  \left(\begin{array}{c}  \sqrt{\Im B}\langle f, g \rangle_{L_2(0,1)}  \\ \\
\sqrt{V(x)}((A_B-z)^{-1}f)(x) \end{array}\right) \in {\mathbb C}\oplus L_2(V^{-1}({\mathbb R}_+)), \]
in which
\[ g(\cdot) = \overline{\left\{\frac{-s_z(\cdot)}{(B+M(z))\Gamma_0 s_z}\right\}} . \] 
A simple calculation shows that for any test vector $e = \left(\begin{array}{c}c \\ u\end{array}\right)\in {\mathbb C}\oplus L_2(V^{-1}({\mathbb R}_+))$,
\[ \left\langle \Gamma(A_B-z)^{-1}f,\left(\begin{array}{c}c \\ u \end{array}\right) \right\rangle_{{\mathbb C}\oplus L_2(V^{-1}({\mathbb R}_+))}
 = \left\langle f, \; g(\cdot) \sqrt{\Im B} \; c + (A_B-z)^{-*}\sqrt{V} u  \right\rangle_{L_2(0,1)}, \]
 so that
 \begin{equation}\label{mm:last2} (\Gamma(A_B-z)^{-1})^*e =  (\Gamma(A_B-z)^{-1})^*\left(\begin{array}{c}c \\ u \end{array}\right)  =  \overline{\left\{\frac{-s_z(\cdot)}{(B+M(z))\Gamma_0 s_z}\right\}} \sqrt{\Im B} \; c
  + (A_B-z)^{-*}\sqrt{V} u. \end{equation}
The inverse  $((\Gamma(A_B-z)^{-1})^*)^{-1}$ can now be found. Denoting $\varphi:= (\Gamma(A_B-z)^{-1})^*\left(\begin{array}{c}c \\ u \end{array}\right)$, we observe that since the term $(A_B-z)^{-*}\sqrt{V}u$ lies in $\mbox{ker}(\Gamma_1+\overline{B}\Gamma_0)$, while
$(\Gamma_1+\overline{B}\Gamma_0)\overline{s_z} = \overline{(\Gamma_1+B\Gamma_0)s_z} = \overline{(B+M(z))\Gamma_0s_z}$, 
we obtain
\[ c = -\frac{1}{\sqrt{\Im B}}(\Gamma_1+\overline{B}\Gamma_0)\varphi. \]
Furthermore, we know that $\displaystyle{\left(-\frac{d^2}{dx^2}+Q-iV-\overline{z}\right)\overline{s_z}=0}$, while
\[ \left(-\frac{d^2}{dx^2}+Q-iV-\overline{z}\right)(A_B-z)^{-*}\sqrt{V} u = \sqrt{V}u, \]
and so
\[ u = \frac{1}{\sqrt{V}}P_{V^{-1}({\mathbb R}_+)}\left(-\frac{d^2}{dx^2}+Q-iV-\overline{z}\right)\varphi, \]
in which $P_{V^{-1}({\mathbb R}_+)}$ is the orthogonal projection from $L_2(0,1)$ to $L_2(V^{-1}({\mathbb R}_+))$.
Finally we arrive at the expression
\begin{equation}\label{mm:last3}  ((\Gamma(A_B-\overline{\lambda_0})^{-1})^*)^{-1}\varphi = \left(\begin{array}{c} -\frac{1}{\sqrt{\Im B}}(\Gamma_1+\overline{B}\Gamma_0)\varphi  \\  \frac{1}{\sqrt{V}}P_{V^{-1}({\mathbb R}_+)}\left(-\frac{d^2}{dx^2}+Q-iV-\lambda_0 \right)\varphi  \end{array}\right)
 =: \left(\begin{array}{c} e_1 \\ e_2 \end{array}\right) =: e. \end{equation}
 }
{
 The formula
 \begin{equation}\label{mm:last3*}  ((\Gamma_*(A_B^*-\overline{\mu_0})^{-1})^*)^{-1}\varphi_* = \left(\begin{array}{c} -\frac{1}{\sqrt{\Im B}}(\Gamma_1+B\Gamma_0)\varphi_*  \\  \frac{1}{\sqrt{V}}P_{V^{-1}({\mathbb R}_+)}\left(-\frac{d^2}{dx^2}+Q+iV-\mu_0 \right)\varphi_*  \end{array}\right)
 =: \left(\begin{array}{c} e_{1,*} \\ e_{2,*} \end{array}\right) =: e_*. \end{equation}
 is proved similarly. }
 Eqn. (\ref{mm:last3}) can be used to compute 
$\Phi\left (\begin{array}{c} v_-\\ \varphi \\v_+\end{array}\right )$
	for any  $\varphi = ((\Gamma(A_B-\overline{\lambda_0})^{-1})^*)e$, $e\in E$,
	$v_-\in L_2(\Rr_-, E_*)$ and $v_+ \in L_2 (\Rr_+,E)$ {using the expression in
	Lemma \ref{lem4.2}}. 	
Similarly, (\ref{mm:last3*}) allows the computation of
$\Phi\left (\begin{array}{c} v_-\\ \varphi \\v_+\end{array}\right )$
	for any  $\varphi = ((\Gamma(A_B-\overline{\mu_0})^{-1})^*)e_*$, $e_*\in E_*$, $v_-\in L_2(\Rr_-, E_*)$ and 
	$v_+ \in L_2 (\Rr_+,E)$ {using the expression in
	Lemma \ref{lem4.3}}. 
	
	We obtain 
	\[ \Phi\left(\begin{array}{c} v_{-} \\  \varphi = ((\Gamma(A_B-\overline{\lambda_0})^{-1})^*)e \\ v_{+} \end{array}\right) = 
	\left(\begin{array}{c} \hat{v}_+(k) \\ \hat{v}_{-}(k)\end{array}\right) 
	 + \frac{i}{\sqrt{2\pi}(k-\lambda_0)}\left(\begin{array}{c} e \\ -S(\lambda_0)e\end{array}\right). \] 
The quantity $S(\lambda_0)e$ is computed using the $2\times 2$ block operator matrix
expression for $S$ in (\ref{eq:s11s12}). We write explicitly only the most complicated quantity, namely
\begin{eqnarray*} S_{22}(\lambda_0)e_2 = e_2 + T_{22}(\lambda_0)e_2 & = & \frac{1}{\sqrt{V(x)}}P_{V^{-1}(\Rr_+)}\left(-\frac{d^2}{dx^2}+Q(x)-iV(x)-\lambda_0\right)\varphi(x) \\ & + &  2i\sqrt{V(x)} \int_0^1 G(x,t;\lambda_0)P_{V^{-1}(\Rr_+)}\left(-\frac{d^2}{dt^2}+Q(t)-iV(t)-\lambda_0\right)\varphi(t)dt,\end{eqnarray*}
in which the Green's function $G$ is 

\[ G(x,t;z) = \frac{\tilde{\phi}_z(\min(x,t))\,\tilde{s}_z(\max(x,t))}{(\overline{B}+\tilde{M}(z))\Gamma_0 \tilde{s}_z}. \]}
	Similarly,
	\[ \Phi\left(\begin{array}{c} v_{-} \\  \varphi_* = ((\Gamma_*(A_B^*-\overline{\mu_0})^{-1})^*)e_* \\ v_{+} \end{array}\right) = 
	\left(\begin{array}{c} \hat{v}_+(k) \\ \hat{v}_{-}(k)\end{array}\right) 
	 + \frac{i}{\sqrt{2\pi}(k-\mu_0)}\left(\begin{array}{c} S^*(\overline{\mu_0})e_* \\ -e_* \end{array}\right), \]
and one may show that
\begin{eqnarray*} S_{22}^*(\overline{\mu_0})e_{2,*} = e_{2,*} + T_{22}^*(\overline{\mu_0})e_{2,*} & = & \frac{1}{\sqrt{V(x)}}P_{V^{-1}(\Rr_+)}\left(-\frac{d^2}{dx^2}+Q(x)+iV(x)-\mu_0\right)\varphi_*(x) \\ & - &  2i\sqrt{V(x)} \int_0^1 \overline{G(x,t;\overline{\mu_0})}P_{V^{-1}(\Rr_+)}\left(-\frac{d^2}{dt^2}+Q(t)+iV(t)-\mu_0\right)\varphi_*(t)dt.\end{eqnarray*}

\section*{Acknowledgements}
The authors would like to thank the referees whose detailed reading of an earlier draft, and constructive criticisms, enabled us
to improve this article substantially.



\begin{thebibliography}{99}

\bibitem{BHS} Behrndt, Jussi; Hassi, Seppo; de Snoo, Henk; {\em Boundary value problems, Weyl functions, and differential operators. }
Monographs in Mathematics, 108. Birkh\"{a}user/Springer (2020). 

\bibitem{BS87}
Birman, M.~Sh.  \and Solomjak, M.~Z.;
\newblock {\em Spectral theory of selfadjoint operators in {H}ilbert space}.
\newblock Mathematics and its Applications (Soviet Series). D. Reidel
  Publishing Co., Dordrecht, 1987.
\newblock Translated from the 1980 Russian original by S. Khrushch\"{e}v and V.
  Peller.
%
%
%
	%
{\bibitem{Bro71} Brodski\u{\i}, M.S.; {\em Triangular and Jordan representations of linear operators.}
Translations of mathematical monographs, v.~32. American Mathematical Society, Providence, R.I., 1971.
}
{\bibitem{Bro78} Brodski\u{\i}, M.S.; {\em Unitary operator colligations and their characteristic functions.} (Russian). 
Uspekhi Mat. Nauk {\bf 44} (1978), no. 4(202), 141--168,256. English translation: Russian Math. Surveys {\bf 33} (1978),
no. 4, 159--191.}
%
\bibitem{BMNW08} Brown, B.~M., Marletta, M., Naboko, S.~\and Wood, I.; {Boundary triplets and $M$-functions for non-selfadjoint operators, with applications to elliptic PDEs and block operator matrices}. {\em J. London Math. Soc.} (2) { 77} (2008), 700--718.
\bibitem{BMNW20}  Brown, B.~M., Marletta, M., Naboko, S.~\and Wood, I.;   The functional model for maximal dissipative operators (translation form): An approach in the spirit of operator knots. Trans. Amer. Math. Soc. 373 (2020), no. 6, 4145--4187. 
\bibitem{CKS18} Cherednichenko, K., Kiselev, A., Silva, L.; { Functional model for extensions of symmetric operators and applications to scattering theory.}
{\em Netw.~Heterog.~Media} {\bf 13} (2018), no. 2, 191--215.
\bibitem{CKS19} Cherednichenko, K., Kiselev, A., Silva, L.;{Functional model for boundary-value problems}. Mathematika {\bf 67} (2021), no. 3, 596--626. 
\bibitem{dB68} {de Branges, L.; {\em Hilbert spaces of entire functions.} Prentice-Hall, Inc., Englewood Cliffs, N.J. 1968.}
\bibitem{DM91}{ Derkach, V.~\and\ Malamud, M.;}  {\em Generalized resolvents and the boundary value problems for Hermitian operators with gaps}. J.~Funct.~Anal. { 95} (1991), 1--95.
\bibitem{DM92}{ Derkach, V.~\and\ Malamud, M.;} {\em Characteristic functions of almost solvable extensions of Hermitian operators.} Ukrainian Mathematical Journal, Volume 44, Issue 4, 379--401, 1992.
\bibitem{GG91} Gorbachuk, V.I. and Gorbachuk, M.L.; {\it  Boundary value problems for operator differential equations}. Kluwer, Dordrecht (1991).
\bibitem{HP57} Hille, E. and Phillips, R.S.: {\em Functional analysis and semi-groups.} Rev.~ed.~American Mathematical Society Colloquium Publications, vol. 31. American Mathematical Society, Providence, R. I., 1957.

\bibitem{Koosis} { Koosis, P.;} {\em Introduction to $H_p$ spaces.} Second edition. 
Cambridge Tracts in Mathematics, 115. Cambridge University Press, Cambridge, 1998.

\bibitem{Lan61}
Langer, H.;
\newblock Ein {Z}erspaltungssatz f\"ur {O}peratoren im {H}ilbertraum.
\newblock {\em Acta Math. Acad. Sci. Hungar.}, 12:441--445, 1961.
\bibitem{LP67} Lax, Peter D. and Phillips, Ralph S.: Scattering theory. Pure and Applied Mathematics, Vol. 26 Academic Press, New York-London 1967.
\bibitem{Liv46} Liv\v{s}ic, M. S.:  On a certain class of linear operators in Hilbert space, {\sl Mat.~Sbornik}, {\bf 19} (1946), no. 2, 239--262.
\bibitem{Liv54} Liv\v{s}ic, M. S.: On spectral decomposition of linear nonself-adjoint operators. (Russian) {\sl Mat.~Sbornik N.S.} {\bf 34} (76), (1954), 145--199.
\bibitem{Liv73} Liv\v{s}ic, M. S.: Operators, oscillations, waves (open systems). Translated from the Russian by Scripta Technica, Ltd. English translation edited by R. Herden. Translations of Mathematical Monographs, Vol. 34. American Mathematical Society, Providence, R.I., 1973.

\bibitem{Nab81} Naboko, S.; {\em A functional model of perturbation theory and its application to scattering theory}. Trudi. Matem.~Inst.~Steklov {\bf 147} (1980); Engl.~transl.~in Proc.~Steklov Inst.~of Math.~(2) (1981), 85--116.
\bibitem{NR01}  Naboko, S. and Romanov, R.; {\em Spectral singularities, Szokefalvi-Nagy-Foias functional model and the spectral analysis of the Boltzmann operator}. In: Recent advances in operator theory and related topics (Szeged, 1999), Oper. Theory Adv. Appl., {\bf 127}, Birkh\"{a}user, Basel (2001), 473--490.

\bibitem{NiessenZettl92} Niessen, H.-D. and Zettl, A.,
\newblock Singular {S}turm-{L}iouville problems: the {F}riedrichs
              extension and comparison of eigenvalues.
\newblock Proc. London Math. Soc. (3), {\bf 64}, 545--578 (1992).

\bibitem{Nik86} Nikolski{\u{\i}}, N.K.; {\em Treatise on the Shift Operator.} Springer, Berlin, 1986.



\bibitem{NV86}
Nikolski{\u{\i}}, N.K. and Vasyunin, V.;
\newblock Notes on two functional models.
\newblock In {\em The Bieberbach conjecture (West Lafayette, Ind., 1985)}, 113--141, Math. Surveys Monogr., 21, Amer. Math. Soc., Providence, RI, 1986.
%

\bibitem{NikVas89} {Nikolski\u{\i}, Nikola\u{\i} K. and Vasyunin, Vasily I.},
     {\em A unified approach to function models, and the transcription
              problem}. Oper. Theory Adv. Appl. {\bf 41}, 405--434, Birkh\"{a}user, Basel (1989).


%
\bibitem{NikVas98}
Nikolski{\u{\i}}, N.K. and Vasyunin, V.;
\newblock Elements of spectral theory in terms of the free function model. {I}.
 {B}asic constructions.
\newblock In {\em Holomorphic spaces ({B}erkeley, {CA}, 1995)}, volume~33 of
 {\em Math. Sci. Res. Inst. Publ.}, pages 211--302. Cambridge Univ. Press,
 Cambridge, 1998.
%


\bibitem{Pav75} Pavlov, B.~S.; {\em Conditions for separation of the spectral components of a dissipative operator.} (Russian) Izv.~Akad.~Nauk SSSR Ser.~Mat. (240) {\bf 39}  (1975), 123--148.
\bibitem{Pav76}
{Pavlov, B.~S.;
\newblock Dilation theory and spectral analysis of nonselfadjoint differential
  operators.
\newblock Mathematical programming and related questions ({P}roc.
               {S}eventh {W}inter {S}chool, {D}rogobych, 1974), {T}heory of
               operators in linear spaces ({R}ussian)
Central. \`Ekonom. Mat. Inst. Akad. Nauk SSSR, Moscow.
\newblock 3--69, 1976.}
\bibitem{Pav77} Pavlov, B.~S.; {\em Selfadjoint dilation of the dissipative Schr\"{o}dinger operator and its resolution in terms of eigenfunctions.} Mat.~Sb. (144) {\bf 102} (1977), 511--536; English transl.: Math.~USSR Sb.~{\bf 31} (1977).
\bibitem{Pav02}
 Pavlov, B.~S.;
{\em  Resonance quantum switch: matching domains}.
\newblock In {\em Surveys in analysis and operator theory ({C}anberra, 2001)},
  volume~40 of {\em Proc. Centre Math. Appl. Austral. Nat. Univ.}, pages
  127--156. Austral. Nat. Univ., Canberra, 2002.
{\bibitem{Ryz07} Ryzhov, V.:
Functional model of a class of non-selfadjoint extensions of symmetric operators.  Operator theory, analysis and mathematical physics,  117--158,Oper.~Theory Adv.~Appl., 174, Birkh\"{a}user, Basel, 2007.}
{\bibitem{Ryz08} Ryzhov, V.: Functional model of a closed non-selfadjoint operator.  {\sl Int.~Eq.~Oper.~Th}, {\bf 60} (2008), 539-- 571.}
\bibitem{Str60} \v{S}traus, A. V.: Characteristic functions of linear operators. (Russian) 
Izv. Akad. Nauk SSSR Ser. Mat.  24,  1960, 43--74.

\bibitem{SFBK10} Sz.-Nagy, B., Foia{\lfhook{s}}, C., Bercovici, H. and K\'erchy, L.:
{\sl    Harmonic analysis of operators on {H}ilbert space},
Second Edition, Springer, New York; 2010.
\bibitem{Tik2004}
Tikhonov, Alexey,
\newblock  Free functional model related to simply-connected domains, in {\em Spectral methods for operators of mathematical physics},
\newblock {Oper. Theory Adv. Appl.}, {\bf 154}, 219--231 (2004).

  
\bibitem{Vas77}
Vasyunin, V. I.,
\newblock  The construction of the {B}. {S}z\"{o}kefalvi-{N}agy and {C}.
              {F}oia\c{s} functional model,
\newblock {\em Investigations on linear operators and the theory of
              functions, VIII}.
\newblock Zap. Nauchn. Sem. Leningrad. Otdel. Mat. Inst. Steklov
              (LOMI) {\bf 73}, 16--23, 229 (1978).

\end{thebibliography}
\end{document}